\documentclass{article}
\usepackage[utf8]{inputenc}
\usepackage{graphicx}
\usepackage{mathtools,amsmath,amssymb,amsfonts,amsthm,upgreek}
\usepackage{mathalfa}
\usepackage{tikz,tikz-cd}
\usetikzlibrary{arrows}
\usetikzlibrary{decorations.markings}
\tikzset{degil/.style={line width=0.5pt,double distance=5pt,
        decoration={markings,
        mark= at position 0.5 with {
              \node[transform shape] (tempnode) {$\backslash\backslash$};
              }
          },
          postaction={decorate}
}
}

\tikzset{
commutative diagrams/.cd,
arrow style=tikz,
diagrams={>=open triangle 45, line width=0.5pt}}

\usetikzlibrary{calc}
\usepackage{color}
\usepackage{url}
\usepackage[backgroundcolor=white,linecolor=yellow,bordercolor=yellow,textcolor=black,textsize=small]{todonotes}
\usepackage[margin=0.5in]{geometry}

\usepackage{enumitem}

\graphicspath{{./figures/}}

\usepackage[english]{babel}

\usepackage{tabularx}

\usepackage{xcolor}

\usepackage{verbatim}
\usepackage{xspace}

\usepackage{amsfonts}
\usepackage{mathtools}
\usepackage{amsmath}

\usepackage{amssymb}
\usepackage{soul} 

\usepackage{setspace}

\usepackage{fullpage}
\usepackage{ifthen}
\usepackage{tikz}
\usetikzlibrary{automata,positioning, arrows}

\theoremstyle{plain}
\newtheorem{thm}{Theorem}
\newtheorem{lemma}{Lemma}
\newtheorem{prop}{Proposition}
\newtheorem{cor}{Corollary}

\theoremstyle{definition}

\newtheorem{defn}{Definition}
\newtheorem{conjecture}{Conjecture}

\theoremstyle{remark}
\newtheorem{myexample}{Example}
\newtheorem{rem}{Remark}
\newtheorem{assumption}{Assumption}

\newenvironment{example}[1][0]
{ 
  \ifthenelse{\equal{#1}{0}}{
  \myexample
}
{ 
  \renewcommand\themyexample{#1}\myexample
  \addtocounter{myexample}{-1}
}
}
{\endmyexample}

\newcommand{\R}{\mathbb{R}}{}
\newcommand{\N}{\mathbb{N}}

\newcommand{\cA}{\mathcal{A}}

\newcommand{\cC}{\mathcal{C}}

\newcommand{\cI}{\mathcal{I}}
\newcommand{\cK}{\mathcal{K}}
\newcommand{\cL}{\mathcal{L}}
\newcommand{\cKL}{\mathcal{KL}}

\newcommand{\cN}{\mathcal{N}}

\newcommand{\cF}{\mathcal{F}}
\newcommand{\cR}{\mathcal{R}}
\newcommand{\cS}{\mathcal{S}}

\newcommand{\cU}{\mathcal{U}}

\newcommand{\cLL}{\text{Lip}}
\newcommand{\sol}{\Sigma}

\newcommand{\wt}{\widetilde}

\newcommand{\dw}{\textit{dw}}
\newcommand{\ad}{\textit{adw}}
\newcommand{\GAS}{\tau\text{-UGB}}
\newcommand{\ES}{\tau\text{-UGEB}_\rho}

\DeclarePairedDelimiterX{\inp}[2]{\langle}{\rangle}{#1, #2}

\newcommand{\restr}[2]{{
  \left.\kern-\nulldelimiterspace 
  #1 
  \vphantom{\big|} 
  \right|_{#2} 
  }}

\title{Converse Lyapunov Results for Stability of Switched Systems with Average Dwell-Time}

\author{Matteo Della Rossa$^*$
  \and  ~Aneel Tanwani \thanks{M.D.R. is with Dipartimento di Scienze Matematiche Informatiche e Fisiche (DMIF), University of Udine, Udine, (Italy). A.T. is with CNRS – LAAS, University of Toulouse, CNRS, 31400 Toulouse, (France)
        {\small matteo.dellarossa@uniud.it}}}

\date{\today}
\begin{document}

\maketitle

\begin{abstract}
This article provides a characterization of stability for switched nonlinear systems under average dwell-time constraints, in terms of necessary and sufficient conditions involving multiple Lyapunov functions. Earlier converse results focus on switched systems with dwell-time constraints only, and the resulting inequalities depend on the flow of individual subsystems. With the help of a counterexample, we show that a lower bound that guarantees stability for dwell-time switching signals may not necessarily imply stability for switching signals with same lower bound on the average dwell-time. Based on these two observations, we provide a converse result for the average dwell-time constrained systems in terms of inequalities which do not depend on the flow of individual subsystems and are easier to check. The particular case of linear switched systems is studied as a corollary to our main result.
\end{abstract}

\section*{Introduction}

Switched systems comprise a family of dynamical subsystems and a switching signal that determines the active subsystem at any given time instant. Early research on the stability of switched systems mostly focused on developing necessary and sufficient conditions using the framework of Lyapunov functions. Due to their peculiar structure, it is natural to develop stability conditions using the Lyapunov functions for individual subsystems as ``building blocks''. Similarly, in case the switched system is unstable for certain class of switching signals, it is natural to look for design of switching signals which stabilize the overall system. All these problems are relatively well-studied by now, but several aspects of these problems are still being investigated in depth to get more insights.

When the individual subsystems share a common equilibrium point, and each subsystem has its own Lyapunov function relative to that equilibrium, then it is natural to look for a class of switching signals for which the stability of the switched system is guaranteed. In such cases, we commonly look for switching signals constrained by imposing a {\em dwell-time} between two consecutive switches, so that for a class of switching signals satisfying certain lower bound on dwell-time, the resulting switched system is uniformly globally asymptotically stable~\cite{Morse}. A generalization of this concept is obtained in terms of {\em average dwell-time} constrained signals which allow for finitely many rapid switches (called chattering bound) while the system respects a lower bound between the switching instants on average~\cite{HesMor99}. Under a compatibility condition on the Lyapunov functions, we can find lower bounds on the {\em average dwell-time} in terms of system data and parameters of the Lyapunov functions that ensure stability \cite{HesMor99}, \cite[Chapter~3]{Lib03}. Several generalizations and refinements of this line of research have been pursued in \cite{Vu07, MullLibe12, MancHaim20, LiuTanwLib} and references therein. A common element of this line of research is that a sufficient condition for stability is provided using multiple Lyapunov functions when the switching signals are constrained in some way. These constraints depend on the assumptions imposed on the vector fields of individual subsystems, and in the case all of them are asymptotically stable, we compute the lower bounds on the average dwell-time of the switching signals for which the overall switched system is asymptotically stable. Numerical algorithms for computing such multiple Lyapunov functions while minimizing the lower bounds on average dwell-time have been proposed in~\cite{HafTan23}. Surprisingly enough, the necessity of such conditions, or the converse Lyapunov results for such systems have not been developed.

On the other hand, in the study of stability conditions over the restricted class of \emph{dwell-time} signals (not in average), researchers have also used  the particular structure of systems to provide bounds (or, tightest possible bounds in some cases) for dwell-time ensuring stability of the overall switched system. In the linear case, the paper~\cite{Wirth2005} provides a complete characterization in terms of necessary and sufficient conditions involving Lyapunov functions for stability under dwell-time constrained switching signals. The necessity of such conditions has to be underlined since it provides a \emph{characterization} of stability for switched linear systems over the class of dwell-time signals, as was already provided in the case of arbitrary switching (see~\cite{MolPya89},~\cite{DayaMar99} for the linear case,~\cite{Mancilla00} for the nonlinear counterpart).  Results based on similar structure/reasoning have been used to get bounds on dwell-time for switched linear systems using linear matrix inequalities (LMIs)~in \cite{GerCol06}, via quadratic Lyapunov functions. Along similar lines, we have LMIs based on discretization in~\cite{AllSha11} and bilinear matrix inequalities using polyhedral norms in~\cite{BlaCol10}. A different numerical technique based on the use of sum-of-squares formulations of the inequalities to compute the dwell-time is proposed in \cite{CheCol12}. More recently, in~\cite{ChiGugPro21},~\cite{DelPaq22} the authors address the question of finding  bounds on stabilizing dwell-times via graph-theory based approaches. A nonlinear version of the aforementioned Lyapunov characterization of stability over dwell-time signals was recently derived in \cite{DellaRos23}. Summarizing, in the case of fixed dwell-time signals, we have a rather mature Lyapunov theory, providing not only sufficient conditions for stability but also~\emph{converse} statements, that leads to completing the picture from analysis perspective.

 In the two broad research directions discussed above (\emph{average} dwell-time and dwell-time stability analysis), one major difference between the conditions proposed in~\cite{Wirth2005} (and the related literature) and the ones used in~\cite{HesMor99} (and subsequent developments) is due to the inequality that relates different Lyapunov functions associated with different modes (see the subsequent~\eqref{eq:JumpCondNORMNonLinear} and~\eqref{eq:JumpNonlinear} for comparison). The inequalities in \cite{Wirth2005} and in its extensions require the explicit knowledge of solutions of individual vector fields, and this solution/flow map is not readily available for many systems. Even in the restricted linear case, the exponential map $A\mapsto e^{At}$ has the limitation of being (componentwise) non-convex, for any $t>0$, and this can lead to numerical restrictions, see the discussion provided in~\cite{AllSha11} and~\cite[Chapter 1]{Geromel23}. On the other hand, the sufficient conditions proposed in \cite{HesMor99}, and later generalized in \cite{LiuTanwLib}, for stability under average dwell-time signals are numerically easier to check because they do not rely on the flow map of individual subsystems. Another elegant aspect of the Lyapunov inequalities in \cite{HesMor99} is that they are independent of the chattering bound and allow us to infer stability over a broader class of switching signals with certain lower bound on average dwell-time and arbitrarily large (but finite) chattering bound. Moreover, in the linear case, these conditions are convex in terms of systems data because they are not directly affected by the exponential matrices associated to the subsystems, and are thus robust in terms of perturbations to system matrices.  
Nevertheless, for the dwell-time constrained switching signals, the conditions in \cite{HesMor99} provide more conservative lower bounds than the ones obtained from \cite{Wirth2005}, as we explicitly observe in this manuscript with the help of an academic example.

 A natural question based on these observations is whether the conditions in \cite{HesMor99} or their nonlinear counterpart in \cite{LiuTanwLib} are also necessary for a tailored notion of stability over the class of average dwell-time signals. 
 The main result of this paper provides an affirmative answer to this converse question when the stability notion under consideration is described by a certain class of $\cKL$ functions. More precisely, this work provides a characterization of a tailored stability notion under \emph{average} dwell-time constraints, with necessary and sufficient conditions in terms of multiple Lyapunov functions. 

The rest of this manuscript is organized as follows: In Section~\ref{sec:SystemClass}, we introduce the considered classes of systems and switching signals, while in Section~\ref{sec:Stability}, we illustrate the studied notions of stability for switched systems, together with some preliminary results and observations. In Section~\ref{sec:MainResult}, we present our main result comprising a \emph{converse} Lyapunov result for stability over the class of average dwell-time signals. In Section~\ref{sec:Linear}, we specialize our results in the more structured linear subsystems case while Section~\ref{sec:conclu} concludes the paper with some closing remarks and an open conjecture.  Some technical arguments are postponed to the Appendix to avoid breaking the flow of the presentation.
\\
\textbf{Notation:}
The set $\R_+:=\{s\in \R\;\vert\;s\geq0\}$ denotes the set of non-negative real numbers.
Given $m,n\in \N$, the classes $\cC(\R^n,\R^m)$ and $\cC^1(\R^n,\R^m)$ denote the sets of continuous and continuously differentiable functions, respectively. The set $\cLL(\R^n,\R^m)$ denotes the set of locally Lipschitz continuous functions from $\R^n$ to $\R^m$, while  $\cLL_0(\R^n,\R^m)$ denotes the set of functions from $\R^n$ to $\R^m$ which are locally Lipschitz on the open set $\R^n\setminus\{0\}$.\\
\emph{Comparison Functions Classes}:   A function $\alpha:\R_+\to \R$ is of \emph{class $\cK$} ($\alpha \in \cK$) if it is continuous, $\alpha(0)=0$, and strictly increasing; it is of \emph{class $\cK_\infty$} if, in addition, it is unbounded. A continuous function $\beta:\R_+\times \R_+\to \R_+$ is of \emph{class $\mathcal{KL}$} if $\beta(\cdot,t)$ is of class $\cK$ for all $s$, and $\beta(r,\cdot)$ is decreasing and $\beta(r,t)\to 0$ as $t\to\infty$, for all $r\in \R_+$. 

\section{Systems Class and Switching Signals}\label{sec:SystemClass}

To describe the class of dynamical systems studied in this paper, we consider a family of vector fields, $f_i : \R^n \to \R^n$, with $i$ belonging to a finite index set $\cI:=\{1,\dots,{\rm m}\}$ for some ${\rm m}\in \N\setminus \{0\}$. We stipulate the following property with each of these vector fields:
\begin{assumption}\label{Asssum:WellPosedness}
For each $i \in \cI$, we suppose that $f_i\in \sol(\R^n)$, where the notation $f_i \in \sol(\R^n)$ is used to say that $f_i\in \cLL_0(\R^n,\R^n)$ and $f_i$ induces a \emph{well-posed dynamical system with an equilibrium at the origin},  i.e., $f_i(0)=0$ and for the \emph{ODE}
 \begin{equation}\label{eq:subsystem}
 \dot x=f_i(x),
 \end{equation}
the existence, uniqueness and forward completeness of forward solutions (in the sense of Carath\'edory, see~\cite[Section I.5]{Hale}), holds for any initial condition~$x_0\in \R^n$. 
\end{assumption}
Note that, in general, $f_i\in \sol(\R^n)$ does not imply that $f_i\in \cC(\R^n,\R^n)$. Consider $f:\R\to \R$ defined by $f(x)=-\text{sgn}(x)$ (with the convention $\text{sgn}(0)=0$), or $g:\R\to \R$ defined by $g(x)=-\frac{1}{x}$ for $x \neq 0$ and $g(0)=0$. It can be verified that $f,g\in \sol(\R)$.\\
Given $\cF=\{f_i\}_{i\in \cI}\subset \sol(\R^n)$, for any $i\in \cI$,  we denote the \emph{(semi-)flow of the $i$-th subsystem} by $\Phi_i:\R_+\times \R^n\to \R^n$, i.e.
\[
\Phi_i(t,x):=\text{solution to~\eqref{eq:subsystem}, starting at $x\in \R^n$ evaluated at time $t\in \R_+$.}
\]
Given $\cF=\{f_i\}_{i\in \cI}\subset \sol(\R^n)$, we consider the \emph{switched system} defined by
\begin{equation}\label{eq:SwitchedSystem}
\dot x(t)=f_{\sigma(t)}(x(t)),\;\;\;x(0)=x_0\in \R^n,\;\;t\in \R_+,
\end{equation}
where $\sigma:\R_+\to \cI$ is an external \emph{switched signal}. More precisely, the switching signals $\sigma$ are selected, in general, among the set $\cS$ defined by
\begin{equation}\label{eq:arbitrarySwitching}
\cS:=\left\{\sigma:\R_+\to \cI\;\vert\;\;\sigma \;\text{piecewise constant and right continuous} \right\}.
\end{equation}
Given a signal $\sigma\in \cS$, we denote the sequence of switching instants, that is, the points at which $\sigma$ is discontinuous,  by $\{t^\sigma_k\}$. Moreover, given $t\geq s$, we define by $N_\sigma(s,t)$ as the number of discontinuity points of $\sigma$ in the interval $[s,t)$. We stress that $t^\sigma_0:=0$ is considered to be a discontinuity point. We note that, for any $\sigma\in \cS$, the set $\{t^\sigma_k \}$ may be infinite or finite; if it is infinite, then it is unbounded.\\
Given  $\sigma \in \cS$,  $x\in \R^n$ and $t\in \R_+$ we denote by $\Phi_\sigma(t,x)$ the \emph{solution} of~\eqref{eq:SwitchedSystem} starting at $x$ and evaluated at $t$ with respect to the switching signal $\sigma\in \cS$.
\\
\emph{Linear Case:}
Given $\cA=\{A_1,\dots, A_{\rm m}\}\subset \R^{n\times n}$, we consider the \emph{switched linear system} defined by
\begin{equation}\label{eq:LinearSwitchedSystem}
\dot x(t)=A_{\sigma(t)}x(t),\;\;\;x(0)=x_0,\;\;t\in \R_+,
\end{equation}
where $\sigma\in \cS$ is again an external switched signal. \\
Given a $V\in \cLL_0(\R^n,\R)\cap \cC(\R^n,\R)$ and $f_i\in \sol(\R^n)$, we denote by $D^+_{f_i}V$ the \emph{Dini-derivative of $V$ with respect to $f_i$}, defined as
\[
D^+_{f_i}V(x):=\limsup_{h\to 0^+}\frac{V(\Phi_i(h,x))-V(x)}{h},\;\;\;\forall\;x\in \R^n.
\]
We recall that, given $x\in \R^n$, if $V$ is continuously differentiable at $x$, then
\[
D^+_{f_i}V(x)=\inp{\nabla V(x)}{f_i(x)}.
\]
In what follows, we introduce the subsets of switching signals considered in this manuscript.
\begin{defn}[(Average) Dwell-Time Signals]\label{defn:AverageSignals}
Given a threshold $\tau>0$ we denote by 
\begin{equation}\label{eq:DweelTimeClass}
\cS_{\dw}(\tau):=\left\{\sigma\in \cS\;\vert\;t^\sigma_{k} -t^\sigma_{k-1}\geq \tau,\;\forall \;t^\sigma_k>0  \right\},
\end{equation}
 the class of $\tau$-\emph{dwell-time switching signals}. \\
Given $N_0\in \N$ and $\tau>0$, we consider the class of $(\tau,N_0)$-\emph{average dwell-time switching signals}, defined by
\begin{equation}\label{eq:AverageFixedN0}
\cS_{\ad}(\tau,N_0):=\left\{\sigma\in \cS\;\vert\;N_\sigma(s,t)\leq N_0+\frac{t-s}{\tau},\;\;\forall \;t\geq s\geq 0 \;\right\};
\end{equation}
in this case, $N_0\in \N$ is referred to as the \emph{chattering bound}.\\
We denote by $\cS^{\infty}_{\ad}(\tau)$ the class of the $\tau$-\emph{average dwell-time switching signals}, defined by
\begin{equation}\label{eq:CLassAverage}
\cS^{\infty}_{\ad}(\tau):=\bigcup_{N_0\in \N}\cS_{\ad}(\tau,N_0).
\end{equation}
Finally, we denote by $\bar \cS(\tau)$ the class of \emph{eventually $\tau$-average signals} defined by
\begin{equation}\label{eq:EventualAverage}
\bar \cS(\tau):=\{\sigma \in \cS\;\vert\;\limsup_{t\to +\infty} N_\sigma(0,t)-\frac{t}{\tau}<+\infty\}.
\end{equation}
\end{defn}
\begin{rem}\label{rem:PropertiesSIgnals}
We provide some remarks and insights regarding the introduced families of signals. Let us fix a $\tau>0$. The following propositions hold:
\begin{enumerate}[leftmargin=*]
\item  For any $N_1\leq N_2$,   $\cS_{\ad}(\tau,N_1)\subset\cS_{\ad}(\tau,N_2)$, or in other words, the set-valued function $n\mapsto\cS_{\ad}(\tau,n)$ is increasing with respect to the partial relation given by the set inclusion;
\item $\cS_{\ad}(\tau,1)=\cS_{\dw}(\tau)$;
\vspace{0.1cm}
\item $\cS^\infty_{\ad}(\tau)\subsetneqq \bar \cS(\tau)$;
\vspace{0.1cm}
 \item Consider $\sigma \in \cS$ periodic of period $T>0$, and suppose without loss of generality that $\lim_{t\to T^-}\sigma(t)\neq \sigma(0)$. Then, $\sigma\in \bar \cS(\tau)$ if and only if $N_\sigma(0,T)\leq \frac{T}{\tau}$. Moreover, this is also equivalent to the property  $\sigma\in \cS_{\ad}(\tau, N_\sigma(0,T))$.
\end{enumerate}
Item~(1) is straightforward. 
Let us prove (2): given $\sigma\in \cS_{\dw}(\tau)$, consider any $0\leq s\leq t$ and suppose $K\tau\leq (t-s)< (K+1)\tau$ for some $K\in \N$. Then $N_\sigma(s,t)\leq K+1\leq 1+\frac{t-s}{\tau}$, proving that $\sigma\in \cS_{\ad}(\tau,1)$. Suppose now $\sigma\in \cS_{\ad}(\tau,1)$ and consider any switching point $t^\sigma_{k}>0$. For all $s\in [t^\sigma_{k-1}, t^\sigma_{k-1}+\tau)$ we have $N_\sigma(t^\sigma_{k-1},s)\leq 1+\frac{s-t^\sigma_{k-1}}{\tau}<2$, and thus $s$ cannot be a discontinuity point (since already $t^\sigma_{k-1}\in [t^\sigma_{k-1},s)$ is). This proves that $t^\sigma_{k}-t^\sigma_{k-1}\geq \tau$, for any $t^\sigma_k>0$, concluding the proof.

Let us prove Item~(3). Let us consider any $\sigma\in \cS^\infty_{\ad}(\tau)$ i.e. there exists $N_0\in \N$ such that $\sigma\in \cS_{\ad}(\tau, N_0)$. We obtain 
\[
N_\sigma(0,t)-\frac{t}{\tau}\leq N_0+\frac{t}{\tau}-\frac{t}{\tau}=N_0,\;\;\forall \,t\in \R_+,
\]
and thus trivially, $\limsup_{t\to +\infty}N_\sigma(0,t)-\frac{t}{\tau}\leq N_0<+\infty$. To see that the inclusion is strict, consider a signal $\sigma:\R_+\to \cI$ such that
\[
\begin{aligned}
\sigma \;&\text{has $k-1$ discontinuity points in the interval }[k^2\tau, (k^2+1) \tau),\;\;\forall \;k\in \N\setminus\{0\},\\
\sigma \;&\text{ is constant otherwise}.
\end{aligned}
\]
It is clear that $\sigma\notin \cS^\infty_{\ad}(\tau)$ since it has an unbounded number of discontinuity points in the intervals of length $\tau$ of the form $[k^2\tau, (k^2+1) \tau)$, as $k\to +\infty$. On the other hand, consider any $t\in \R_+$ and suppose $t\in [k^2\tau, (k+1)^{2}\tau)$, for any $k\in \N$.  We then have
\[
N_{\sigma}(0,t)-\frac{t}{\tau}\leq \sum_{h=1}^{h=k}(h-1)-k^2\leq \frac{k(k-1)}{2}-k^2\leq \frac{-k^2-k}{2}\leq 0.
\]
By arbitrariness of $t\in \R_+$, we conclude that $\limsup_{t\to +\infty} N_\sigma(0,t)-\frac{t}{\tau}<+\infty$, i.e., $\sigma\in \bar \cS(\tau)$.

We now prove Item~(4),  let us suppose that $N_\sigma(0,T)\leq \frac{T}{\tau}$. It can be easily proved by induction that, for any $t_k=kT$ with $k\in \N$ we have $N_\sigma(0,t_k)=kN_\sigma(0,T)$. Then consider any $t\in \R_+$ and suppose $t\in [KT,(K+1)T)$ for some $K\in \N$, then,
\[
N_\sigma(0,t)-\frac{t}{\tau}\leq (K+1)N_\sigma(0,T)-\frac{KT}{\tau}\leq(K+1)\frac{T}{\tau}-K\frac{T}{\tau}=\frac{T}{\tau}. 
\]
Since the bound is independent of $K$, this implies that $\limsup_{t\to +\infty}N_\sigma(0,t)-\frac{t}{\tau}<+\infty$ and thus $\sigma\in \bar \cS(\tau)$. For the converse implication, let us suppose $N_\sigma(0,T)>\frac{T}{\tau}$ and thus consider $\varepsilon>0$ such that $N_\sigma(0,T)=\frac{T}{\tau}+\varepsilon$. Let us consider $t_k=kT$ for all $k\in \N$, it follows that
\[
N_\sigma(0,t_k)-\frac{t_k}{\tau}=kN_\sigma(0,T)-\frac{kT}{\tau}=k(\frac{T}{\tau}+\varepsilon)-\frac{kT}{\tau}=k\varepsilon.
\]
We have thus proved that $\lim_{k\to +\infty}N_\sigma(0,t_k)-\frac{t_k}{\tau}=+\infty$, proving that $\sigma \notin \bar \cS(\tau)$.

Finally, it is easy to see that, if $N_\sigma(0,T)\leq \frac{T}{\tau}$ then $N_\sigma(s,t)\leq \frac{t-s}{\tau}+N_\sigma(0,T)$ for any $0\leq s\leq t$, implying $\sigma\in \cS_{\ad}(\tau, N_\sigma(0,T))$. Since we have already proved that $\cS^\infty_{\ad}(\tau)\subset \bar \cS(\tau)$, the desired assertion follows.
\end{rem}

\section{Stability over Classes of Switching Signals: Review and First Results}\label{sec:Stability}
In this section we recall and review classical concepts of stability for switched systems with respect to classes of switching signals, providing some discussion and first results.
We introduce here the concept of \emph{uniform} stability with respect to a class of switching signals for systems as in~\eqref{eq:SwitchedSystem}.
\begin{defn}[Stability Notions for a Given Class]\label{defn:StabilityGivenClass}
Given a class $\wt \cS\subseteq \cS$ and $\cF=\{f_i\}_{i\in \cI}\subset \sol(\R^n)$,  system~\eqref{eq:SwitchedSystem} is said to be \emph{uniformly globally asymptotically stable} (UGAS) with respect to (w.r.t.) $\wt\cS$ if there exists $\beta\in \cKL$ such that
\begin{equation}\label{eq:KLBounds}
|\Phi_\sigma(t,x)|\leq \beta(|x|,t),\;\; \;\;\forall\;\sigma\in \wt \cS,\;\forall x\in \R^n,\;\forall \;t\in \R_+.
\end{equation}
Given $\rho>0$, system~\eqref{eq:SwitchedSystem} is said to be \emph{uniformly globally exponentially stable with decay $\rho$} ($\text{UGES}_\rho$) w.r.t. $\wt \cS$ if there exists $M\geq 0$ such that
\begin{equation}\label{eq:ExponentialStability}
|\Phi_\sigma(t,x)|\leq Me^{-\rho t}|x|,\;\;\forall\;\sigma\in \wt \cS,\;\forall x\in \R^n,\;\forall \;t\in \R_+.
\end{equation}
The supremum over the $\rho>0$ for which~\eqref{eq:ExponentialStability} is satisfied for some $M\geq 0$ is called the $\wt \cS$-\emph{exponential decay rate}, and it is denoted by $\rho_{\wt \cS}(\cF)$. 
\end{defn}
We report from the literature, in a condensed form, the main characterization result for UGAS of~\eqref{eq:SwitchedSystem} (and $\text{UGES}_\rho$ of~\eqref{eq:LinearSwitchedSystem}) with respect to dwell-time switching signals, i.e., $\wt \cS=\cS_{\dw}(\tau)$, for a given $\tau>0$.
\begin{prop}[Lyapunov characterization for $\cS_{\dw}(\tau)$]\label{prop:ConverseDwellTime}
Consider any $\tau>0$ and $\cF=\{f_1,\dots, f_{\rm m}\}\subset \sol(\R^n)$. System~\eqref{eq:SwitchedSystem} is UGAS w.r.t. $\cS_{\dw}(\tau)$ if and only if there exist $V_1,\dots, V_{\rm m}\in \cLL_0(\R^n,\R)$ and $\alpha_1,\alpha_2\in \cK_\infty$ such that
\begin{subequations}
\begin{align}
\alpha_1(|x|)\leq V_i(x)\leq \alpha_2(|x|),\;\;\;\;\;&\forall i\in \cI,\forall x\in \R^n,\label{eq:SandwichDwellTimaNonlinear}\\
D^+_{f_i}V_i(x)\leq - V_i(x),\;\;\;\;\;&\forall i\in \cI,\forall x\in \R^n,\label{eq:ContinuousDecreasinDwellTimaNonlinear}\\
V_j(\Phi_i(\tau,x))\leq e^{-\tau}V_i(x),\;\;\;\;\;&\forall (i,j)\in \cI^2,\;\;\forall x\in \R^n. \label{eq:JumpCondNORMNonLinear}
\end{align}
\end{subequations}
\emph{(Linear Case):} Consider $\cA=\{A_1,\dots, A_{\rm m}\}\subset \R^{n\times n}$; given $\tau>0$ and $\rho>0$, system \eqref{eq:LinearSwitchedSystem} is $\text{UGES}_\rho$ w.r.t. $\cS_\dw(\tau)$ if and only if there exist norms $v_1,\dots, v_{\rm m}:\R^n\to \R$ such that
\begin{subequations}\label{eq:LinearDwellTimeINequalities}
\begin{align}
D^+_{A_i}v_i(x)\leq -\rho v_i(x),\;\;\;\;\;&\forall i\in \cI,\forall x\in \R^n,\label{eq:ContinuousDecreasinNORM}\\
v_j(e^{\tau A_i}x)\leq e^{-\rho \tau}v_i(x),\;\;\;\;\;&\forall (i,j)\in \cI^2,\;\;\forall x\in \R^n. \label{eq:JumpCondNORM}
\end{align}
\end{subequations}
\end{prop}
The proof of this proposition, for the linear case, is provided in~\cite{Wirth2005}, see also~\cite{ChiGugPro21}. For the proof of the direct extension to the non-linear case, see~\cite{DellaRos23}.

Summarizing, Definition~\ref{defn:StabilityGivenClass} provides notions of stability which are \emph{uniform} over a given class $\wt \cS\subseteq\cS$.  In the case of $\cS_{\dw}(\tau)$, for some $\tau>0$, there already exists a complete Lyapunov characterization of such stability notions, as illustrated in Proposition~\ref{prop:ConverseDwellTime}. 
Similar results can be found for related classes of signals (signals with lower bound on the length of intervals, signals with graph-based constraints, etc.), see for example~\cite{DellaRos23,ChiMas17,ChiGugPro21,ProKam23}.

On the other hand, for the class of switching signals $\cS^\infty_{\ad}(\tau)$ defined in~\eqref{eq:CLassAverage}, Definition~\ref{defn:StabilityGivenClass} is somehow too restrictive in the sense that it requires a single $\cK\cL$ function that works for $\cS_\ad(\tau, N_0)$ for all values of $N_0$. This is formally illustrated in the following lemma.
\begin{lemma}\label{lemma:ImpyArbitrayStability}
 Given $\tau>0$ and $\cF=\{f_i\}_{i\in \cI}\subset \sol(\R^n)$, system~\eqref{eq:SwitchedSystem} is  UGAS w.r.t. $\cS^\infty_{\ad}(\tau)$ if and only if  it is UGAS w.r.t. $\cS$ (a.k.a. arbitrary switching stability). Similarly, given $\rho>0$, system~\eqref{eq:SwitchedSystem} is  $\text{UGES}_\rho$ w.r.t. $\cS^\infty_{\ad}(\tau)$ if and only if  it is $\text{UGES}_\rho$ w.r.t. $\cS$.
\end{lemma}
\begin{proof}
If  system~\eqref{eq:SwitchedSystem} is UGAS w.r.t. $\cS$ then it is UGAS w.r.t. $\cS^\infty_{\ad}(\tau)$, since  $\cS^\infty_{\ad}(\tau)\subset \cS$. Let us then suppose  that system~\eqref{eq:SwitchedSystem} is UGAS w.r.t. $\cS^\infty_{\ad}(\tau)$, i.e., there exists $\beta \in \cKL$ such that
\[
|\Phi_\sigma(t,x)|\leq \beta(|x|,t)\;\; \;\;\forall\;\sigma\in  \cS_{\ad}^\infty(\tau),\;\forall x\in \R^n,\;\forall \;t\in \R_+.
\]
Consider any $\sigma\in \cS$ and any $T>0$, we have that there exists a $N_0\in \N$ such that 
\[
\sigma_{\vert_{[0,T)}}\in \cS_{\ad}(\tau,N_0)_{\vert_T},
\]
where, given $\wt \cS\subseteq \cS$, we denote by $\wt \cS\vert_T$ the set of restrictions of signals in $\wt \cS$ over the interval $[0,T)$, i.e.
\[
\wt \cS_{\vert_T}:=\{\gamma:[0,T)\to \mathcal{I}\;\vert\; \exists \,\sigma\in \wt \cS\text{ s.t. } \gamma=\sigma_{\vert_{[0,T)}}\}.
\] 
As a (conservative) case, we can choose $N_0=N_\sigma(0,T)<+\infty$.
In other words, there exists $\wt \sigma\in \cS^\infty_{\ad}(\tau)$ such that $\sigma(t)=\wt \sigma(t)$ for all $t\leq T$.
We thus have
\[
|\Phi_\sigma(t,x)|=|\Phi_{\wt \sigma}(t,x)|\leq \beta(|x|,t),\;\;\forall \,x\in \R^n,\; \forall\, t\leq T.
\]
By arbitrariness of $\sigma \in \cS$ and $T>0$, we conclude that UGAS w.r.t. $\cS^\infty_{\ad}(\tau)$ implies UGAS w.r.t. $\cS$.
\\
The $\text{UGES}_\rho$ case follows a completely equivalent argument, that it is thus not explicitly reported. 
\end{proof}
Lemma~\ref{lemma:ImpyArbitrayStability} suggests to introduce a non-uniform notion of stability in order to take into account signals in $\cS^\infty_{\ad}(\tau)$ (and therefore also in $\bar \cS(\tau)$). In what follows we introduce a somehow ``strong'' notion of boundedness, that takes into account the number of discontinuity points of any signal, up to the considered time.

\begin{defn}[Jump Dependent Boundedness Notions]\label{defn:strongNotion}
Consider any $\tau>0$ and $\cF=\{f_i\}_{i\in \cI}\subset \sol(\R^n)$, system~\eqref{eq:SwitchedSystem} is~\emph{$\tau$-dependent uniformly globally bounded} (or, $\GAS$ for short) if there exist $\eta_1,\eta_2\in \cK_\infty$ and $\alpha\geq 0$  such that
\begin{equation}\label{eq:NonlinearRelaxation}
|\Phi_\sigma(t,x)|\leq  \eta_1(e^{\alpha \tau N_\sigma(0,t)}e^{-(1+\alpha) t} \eta_2(|x|))\;\;\,\forall \sigma\in \cS,\;\forall x\in \R^n,\;\;\forall \,t\in \R_+.
\end{equation}
 Similarly, given $\rho>0$, we say that  system~\eqref{eq:SwitchedSystem} is~\emph{$\tau$-dependent uniformly globally exponentially bounded with decay $\rho$} ($\ES$ for short) if there exist $M>0$ and $\alpha\geq 0$ such that
\begin{equation}\label{eq:LinearRElaxation}
|\Phi_\sigma(t,x)|\leq M\,e^{\alpha \tau N_\sigma(0,t)}e^{-(\rho+\alpha) t}|x|,\;\;\;\,\forall \sigma\in \cS,\;\forall x\in \R^n,\;\;\forall \,t\in \R_+.
\end{equation}
\end{defn}
The notions of $\GAS$ and $\ES$ introduced in Definition~\ref{defn:strongNotion} are somehow peculiar: the defining inequalities~\eqref{eq:NonlinearRelaxation} and~\eqref{eq:LinearRElaxation} indeed provide bounds on the norm of solutions to~\eqref{eq:SwitchedSystem} \emph{for any signal} $\sigma\in \cS$, and such bounds explicitly depend on the number of switches/discontinuity points. These bounds lead to convergence of solutions only for $\sigma\in \cS$ for which the term $N_\sigma(0,t)$ is not significantly large with respect to $\frac{t}{\tau}$, at least when $t$ approaches $+\infty$. More specifically, such convergence property holds true for eventually $\tau$-average signals, as defined in~\eqref{eq:EventualAverage}. Moreover, such notions imply \emph{uniform} global asymptotic stability with respect to the classes $\cS_{\ad}(\tau,N_0)$, for any fixed $N_0\in \N$. Before formally proving these claims, we will discuss some preliminary concepts
related to the stability notions introduced.

\begin{rem}[Arbitrary switching stability]\label{rem:ArbitraryStability}
We note that in the case $\alpha=0$, inequality~\eqref{eq:NonlinearRelaxation} is equivalent to UGAS w.r.t. $\cS$. This can be proved recalling the Sontag's $\cKL$-lemma (see for example~\cite[Proposition 7]{Son98} or \cite[Lemma 3]{TeelPraly2000}). Similarly, in the case $\alpha=0$,~\eqref{eq:LinearRElaxation} is equivalent to $\text{UGES}_\rho$ w.r.t. $\cS$.
Uniform (exponential) stability on the class of arbitrary switching signals $\cS$ is well-studied, and it already has its Lyapunov characterization, via \emph{common Lyapunov functions}, (see~\cite{Mancilla00} for the non-linear case and~\cite{MolPya89,DayaMar99} for the linear case).
For this reason, the case $\alpha=0$ in~\eqref{eq:NonlinearRelaxation} and~\eqref{eq:LinearRElaxation}  can be considered as a trivial case. 
\end{rem}

Moreover, it can be seen that the choice of the term ``$-1$'' in inequality~\eqref{eq:NonlinearRelaxation} is arbitrary, and can be replaced by $-\lambda$ for any $\lambda>0$ . We formally prove this property in the following statement.
\begin{lemma}\label{lemma:SpeedTransformation}
Given any $\tau>0$, if system~\eqref{eq:SwitchedSystem} is  $\GAS$ then, \emph{for any $\lambda>0$}, there exist $\wt\eta_1,\wt\eta_2\in \cK_\infty$ and $\wt \alpha\geq 0$  such that
\begin{equation}
|\Phi_\sigma(t,x)|\leq  \wt \eta_1(e^{\wt \alpha \tau N_\sigma(0,t)}e^{-(\lambda+\wt \alpha) t} \wt \eta_2(|x|))\;\;\,\forall \sigma\in \cS,\;\forall x\in \R^n,\;\;\forall \,t\in \R_+.
\end{equation}
\end{lemma}
\begin{proof}
Let us suppose that \eqref{eq:SwitchedSystem} is  $\GAS$, and consider $\eta_1,\eta_2\in \cK_\infty$ and $\alpha\geq 0$ such that inequality~\eqref{eq:NonlinearRelaxation} holds. Consider an arbitrary $\lambda>0$. In the case $\alpha=0$, the statement trivially follows by choosing~$\wt \alpha=0$, $\widetilde \eta_1(r)=\eta_1(r^{\frac{1}{\lambda}})$  and $\widetilde \eta_2(r)=(\eta_2(r))^{\lambda}$ for any $r\in \R_+$. If $\alpha>0$, we choose $\wt \alpha>0$ such that the following equation is satisfied:
\[
\frac{\wt \alpha}{\lambda+ \wt \alpha}=\frac{\alpha}{1+\alpha},
\]
and let us call $\theta:=\frac{\lambda+ \wt \alpha}{1+\alpha}=\frac{\wt \alpha}{\alpha}>0$. Then, we have
\[
\begin{aligned}
|\Phi_\sigma(t,x)| &\leq \eta_1(e^{\alpha \tau N_\sigma(0,t)}e^{-(1+\alpha) t} \eta_2(|x|))= \eta_1\left(\left( (e^{\alpha \tau N_\sigma(0,t)}e^{-(1+\alpha) t} \eta_2(|x|)\,)^\theta\right)^{\frac{1}{\theta}} \right)
\\&= \eta_1\left(\left( e^{\wt \alpha \tau N_\sigma(0,t)}e^{-(\lambda+ \wt\alpha) t} (\eta_2(|x|))^\theta\right)^{\frac{1}{\theta}} \right),
\end{aligned}
\]
and we can thus conclude by letting $\wt \eta_1(r) :=\eta_1(r^{\frac{1}{\theta}})$ and $\wt \eta_2(r):=(\eta_2(r))^\theta$.
\end{proof}

The notion introduced in Definition~\ref{defn:strongNotion} has important consequences on stability/boundedness of solutions to~\eqref{eq:SwitchedSystem}, when considering switching signals in the classes $\cS_{\ad}(\tau,N_0)$ and $\bar \cS(\tau)$, as we report in the following statement.
\begin{lemma}[Stability properties]\label{Lemma:StrongConsequences}
Given $\tau>0$ and $\rho>0$ suppose that system~\eqref{eq:SwitchedSystem} is $\GAS$ (resp. $\ES$). Then the following propositions hold:
\begin{enumerate}[leftmargin=*]
\item System~\eqref{eq:SwitchedSystem} is UGAS (resp. $\text{UGES}_\rho$) w.r.t. $\cS_{\ad}(\tau,N_0)$, \emph{for all $N_0\in \N$}.
\item For all $\sigma\in \bar \cS(\tau)$, there exists $\beta_\sigma\in \cKL$ (resp. $M_\sigma\geq 0$) such that
\[
\begin{aligned}
&|\Phi_\sigma(t,x)|\leq \beta_\sigma(|x|,t),\;\;\;\;\forall \;x\in \R^n,\forall t\in \R_+,\\
\text{(resp.)}\;\;&|\Phi_\sigma(t,x)|\leq M_\sigma e^{-\rho t}|x|,\;\;\forall \;x\in \R^n,\forall t\in \R_+.
\end{aligned}
\]
\end{enumerate}
\end{lemma}
\begin{proof}
Let us prove Item~\emph{(1)} first. Suppose that  system~\eqref{eq:SwitchedSystem} is $\GAS$ and thus consider $\eta_1,\eta_2\in \cK_\infty$ and $\alpha\geq 0$ as in Definition~\ref{defn:strongNotion}. Consider any $N_0\in \N$ and any $\sigma \in \cS_{\ad}(\tau, N_0)$. For any $x\in \R^n$ and any $t\in \R_+$ we have
\[
|\Phi_\sigma(t,x)|\leq \eta_1(e^{\alpha \tau N_\sigma(0,t)}e^{-(1+\alpha) t} \eta_2(|x|))\leq  \eta_1(e^{\alpha \tau (N_0+\frac{t}{\tau})}e^{-(1+\alpha) t} \eta_2(|x|))=\eta_1(e^{\alpha  N_0\tau}e^{-t} \eta_2(|x|)).
\]
Defining $\beta_{N_0}(r,t):=\eta_1(e^{\alpha  N_0\tau}e^{-t} \eta_2(r))$, for all $N_0\in \N$ we have
\[
|\Phi_\sigma(t,x)|\leq \beta_{N_0}(|x|,t),\;\;\;\forall\;\sigma\in  \cS_{\ad}(\tau,N_0),\;\forall x\in \R^n,\;\forall \;t\in \R_+,
\]
concluding the proof. The exponential case, i.e. that $\ES$ implies $\text{UGES}_\rho$ on $\cS_{\ad}(\tau,N_0)$, for all $N_0\in \N$ is similar and thus avoided.\\
Next, let us prove Item~\emph{(2)}. Consider any $\sigma \in \bar \cS(\tau)$, i.e. we suppose that $\limsup_{t\to +\infty} N_{\sigma}(0,t)-\frac{t}{\tau}=C<+\infty$. By definition of the limit superior, given any $\varepsilon>0$, there exists $T_\varepsilon>0$ such that 
\begin{equation}
N_\sigma(0,t)\leq C+\varepsilon+\frac{t}{\tau},\;\;\;\forall \;t\geq T_\varepsilon.
\end{equation}
Let us denote by $N_1=N_\sigma(0,T_\varepsilon)$, then by definition of $\GAS$, we have
\begin{equation}\label{eq:TechIneq1}
|\Phi_\sigma(t,x)|\leq \eta_1(e^{\alpha \tau N_\sigma(0,t)}e^{-(1+\alpha) t} \eta_2(|x|))\leq  \eta_1(e^{\alpha  N_1\tau}e^{-(1+\alpha)t} \eta_2(|x|)),\;\;\forall \;x\in \R^n,\;\;\forall t\leq T_\varepsilon.
\end{equation}
Moreover,
\begin{equation}\label{eq:TechIneq2}
\begin{aligned}
|\Phi_\sigma(t,x)|&\leq \eta_1(e^{\alpha \tau N_\sigma(0,t)}e^{-(1+\alpha) t} \eta_2(|x|))\leq  \eta_1(e^{\alpha \tau(C+\varepsilon+\frac{t}{\tau})}e^{-(1+\alpha)t} \eta_2(|x|))\\&\leq \eta_1(e^{\alpha \tau(C+\varepsilon)}e^{-t} \eta_2(|x|))
,\;\;\forall \;x\in \R^n,\;\;\forall t\geq  T_\varepsilon.
\end{aligned}
\end{equation}
Now consider $R=\max\{N_1,\,C+\varepsilon\}$ and define $\beta_\sigma(r,t):=\eta_1(e^{\alpha \tau R}e^{-t} \eta_2(r))$. Merging~\eqref{eq:TechIneq1} and~\eqref{eq:TechIneq2} we obtain
\[
|\Phi_\sigma(t,x)|\leq \beta_\sigma(|x|,t)\;\;\forall \;x\in \R^n,\;\;\forall \;t\in \R_+,
\]
concluding the proof. Again, the exponential case can be proved by similar arguments.
\end{proof}
 In Lemma~\ref{Lemma:StrongConsequences} we have shown that $\GAS$ and $\ES$ imply the classical UGAS and $\text{UGES}_\rho$  properties w.r.t. $\cS_{\ad}(\tau,N_0)$, \emph{for any fixed} $N_0\in \N$. Moreover, we have seen that they also imply asymptotic (exponential) stability for system~\eqref{eq:SwitchedSystem} when a $\sigma \in \bar \cS(\tau)$ is fixed a priori.
In the next section, we will show how these notions have a direct and neat Lyapunov characterization in terms of multiple Lyapunov functions, somehow mimicking the result provided in Proposition~\ref{prop:ConverseDwellTime} for the UGAS and $\text{UGES}_\rho$  w.r.t. $\cS_{\dw}(\tau)$.

Before providing the aforementioned converse result, in the next subsection we discuss the relations between Definition~\ref{defn:StabilityGivenClass} for $\cS_{\dw}(\tau)$ and Definition~\ref{defn:strongNotion}.

\subsection{$\text{UGES}_\rho$ w.r.t. $\cS_{\dw}(\tau)$ does not imply $\text{UGES}_\rho$ w.r.t. $\cS_{\ad}(\tau,2)$}\label{subsec:Counter}
 In switched systems literature, stability with respect to dwell-time signals $\cS_{\dw}(\tau)$ and stability with respect to \emph{average} dwell-time signals $\cS_{\ad}(\tau,N_0)$ for any $N_0\in \N$ are often interchanged and considered to be qualitatively the same. From a general point of view, this is justified in a ``stabilization setting'': switching among a finite set of \emph{exponentially stable} subsystems will preserve stability, if the switching is slow enough (absolutely or in average). This was the philosophy behind the earlier references~\cite{Morse,HesMor99,Lib03} and related results.
 On the other hand, the relations between these notions of stability  have not been completely analyzed from a theoretical point of view, as far as we know. We provide a first step into this analysis in this short subsection.
\begin{figure}[b!]
\centering
\includegraphics[width=.68\linewidth, height =.46\linewidth]{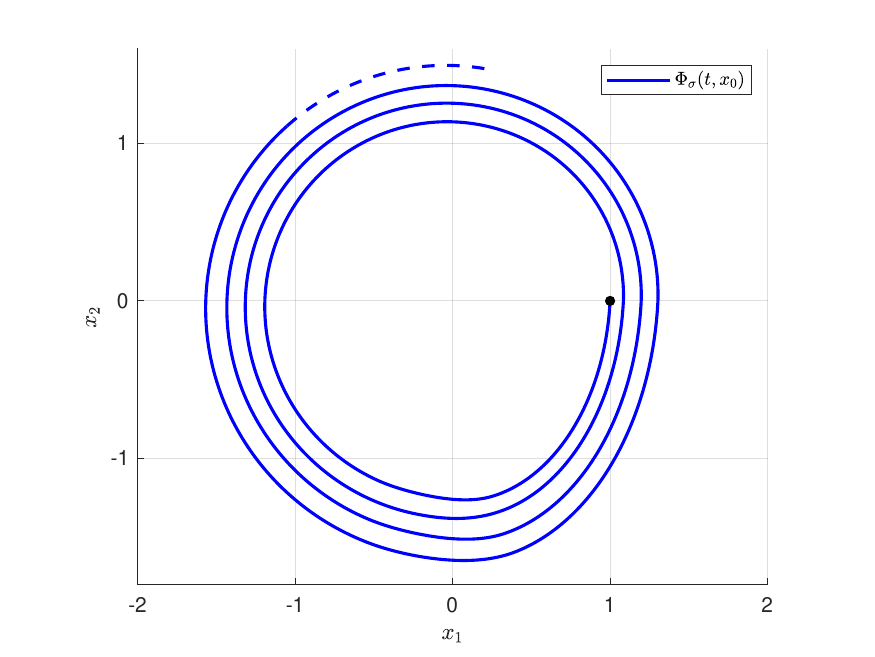}
\caption{Considering the switching signal $\sigma\in \cS_{\ad}(\tau,2)$ in~\eqref{eq:SignalDestabilizing}, and the initial condition $x_0=[1,0]^\top$, we represent the corresponding solution $\Phi_\sigma(\cdot,x_0):\R_+\to \R^2$, highlighting the fact that it is unbounded, i.e.,   $\lim_{t\to +\infty}|\Phi_\sigma(t,x_0)|=+\infty$.}
    \label{fig:Example}
\end{figure}

We start by recalling from Remark~\ref{rem:PropertiesSIgnals} that, given any $\tau>0$,  we have $\cS_{\dw}(\tau)=\cS_{\ad}(\tau,1)\subset \cS_{\ad}(\tau,N_0)$ for any $N_0\geq 1$.
It is thus clear, from Definition~\ref{defn:StabilityGivenClass}, that UGAS (resp. $\text{UGES}_\rho$) w.r.t. $\cS_{\ad}(\tau,N_0)$ for any $N_0\in \N$ implies UGAS (resp. $\text{UGES}_\rho$) w.r.t. $\cS_{\dw}(\tau)$.

Via an explicit numerical example, we now prove that the converse implication does not hold.
More precisely, we provide a switched linear system that is $\text{UGES}_\rho$ w.r.t. $\cS_{\dw}(\tau)$ (for some $\rho>0$ and some $\tau>0$) but unstable on $\cS_{\ad}(\tau,2)$. The system data is obtained by modifying a benchmark example provided in~\cite{Lib03}.
\begin{example} \label{ex:ExampleInstability}
Consider $\cA:=\{A_1,A_2\}\subset \R^{2\times 2}$ defined by
\[
A_1=\begin{bmatrix}
-0.1 & 1\\ -2 &-0.1
\end{bmatrix},\;\;\;A_2=\begin{bmatrix}
-0.03 & 1\\ -1 &-0.03
\end{bmatrix}.
\]
Using the sufficient conditions provided by Proposition~\ref{prop:ConverseDwellTime}, we are able to prove that~\eqref{eq:LinearSwitchedSystem} is $\text{UGES}_\rho$ w.r.t. $\cS_{\dw}(\tau)$, with $\rho=0.001$ and $\tau=\tau_\star:=2.1$. More precisely, it can be seen that there exist $P_1,P_2\in \text{Sym}(\R^{2\times 2})$ such that
\[
\begin{aligned}
I_2\preceq P_i&\preceq 100 I_2,\;\;\;\forall\;i\in \{1,2\},\\
P_iA_i+A_i^\top P_i+2\rho P_i&\prec 0,\;\;\;\forall\;i\in  \{1,2\},\\
e^{A_j^\top\tau_\star} P_j e^{A_j\tau_\star}- e^{-2\rho\tau_\star}P_i&\prec 0,\;\;\;\forall\;(i,j)\in  \{1,2\}^2,
\end{aligned}
\]
where $I_2\in \R^{2\times 2}$ denotes the identity matrix.
This implies that conditions~\eqref{eq:LinearDwellTimeINequalities} in Proposition~\ref{prop:ConverseDwellTime} are satisfied  considering the \emph{quadratic norms} $v_i:\R^2\to \R$ defined by $v_i(x)=\sqrt{x^\top P_i x}$ for $i\in \{1,2\}$.

 On the other hand, we now see that the system is unstable on $\cS_{\ad}(\tau_\star, 2)$, by providing a destabilizing signal $\sigma\in \cS_{\ad}(\tau_\star, 2)$. 
It can be seen that, for any $t\in \R_+$, it holds that
\[
e^{A_1t}=e^{-0.1t}\begin{bmatrix}
\cos(\sqrt{2}t) & \frac{\sqrt{2}}{2}\sin(\sqrt{2}t)\\ -\sqrt{2}\sin(\sqrt{2}t) & \cos(\sqrt{2}t)
\end{bmatrix},\;\;\;e^{A_2t}=e^{-0.03t}\begin{bmatrix}
\cos(t) & \sin(t)\\ -\sin(t) & \cos(t)
\end{bmatrix}.
\]
This suggests to consider $t_1,t_2\in \R_+$ defined by
\[
t_1:=\frac{\pi}{2\sqrt{2}}, \;\;\text{ and }\;\; t_2:=\frac{3\pi}{2},
\]
and to consider the periodic signal $\sigma:\R_+\to \{1,2\}$ of period $T:=t_1+t_2$ defined on the interval $[0,\,T)$  by
\begin{equation}\label{eq:SignalDestabilizing}
\sigma(t)=\begin{cases}
1,\;\;\;\text{if }t\in [0,t_1),\\
2,\;\;\;\text{if }t\in [t_1, T).
\end{cases}
\end{equation}
Intuitively, the idea behind the construction of this switching signal is the following: $t_1$ represents the time in which the solutions of $\dot x=A_1x$ starting on the axes span a $\frac{\pi}{2}$ turn in the state space (clockwise) while  $t_2$  is the time that the solutions of $\dot x=A_2x$  employ to span a $\frac{3\pi}{2}$ turn (clockwise), see Figure~\ref{fig:Example}.
The state-transition matrix of the system $\dot x(t)=A_{\sigma(t)}x(t)$ at time $T$ (see~\cite[Section 4.6]{khalil2002nonlinear}) is then given by
\[
e^{A_2t_2}e^{A_1t_1}=e^{-0.1t_1-0.03t_2}\begin{pmatrix}
\sqrt{2} & 0 \\ 0 & \frac{\sqrt{2}}{2}
\end{pmatrix}.
\]
Choosing as initial condition $x_0=[1,\,0]^\top$, it can be seen that the corresponding solution is exploding, indeed
\[
\lim_{k\to +\infty}|\Phi_{ \sigma}(kT,x_0)|=\lim_{k\to \infty}e^{-k(0.1t_1+0.03t_2)}(\sqrt{2})^k |x_0|>\lim_{k\to +\infty}(1.09)^k=+\infty.
\]
Since $N_\sigma(0,T)=2<2.773\approx \frac{T}{\tau_\star}$, from Item~(4)  of Remark~\ref{rem:PropertiesSIgnals} we have that $\sigma\in \cS_{\ad}(\tau_\star, 2)$. 

By providing an explicit diverging solution, we have thus proved that the system is unstable w.r.t.~switching signals in $\cS_{\ad}(\tau_\star,2)$. 
\end{example}
Summarizing, we have shown that uniform exponential stability w.r.t.~$\cS_{\dw}(\tau)$ does not always imply uniform exponential stability w.r.t. $\cS_{\ad}(\tau, N_0)$, even for planar linear subsystems and for $N_0=2$, and thus in particular it cannot imply $\ES$ as introduced in Definition~\ref{defn:strongNotion}. 

On the other hand, it will be shown in Section~\ref{sec:Linear} that for the family of linear systems $\{A_1,A_2\}$ considered in this example, we can find $\tau^{\star}>\tau_\star$ such that the system is $\text{UGES}_\rho$ with respect to $\cS_{adw}(\tau^{\star},N_0)$ for any $N_0 \in \N$. But the system is clearly not UGAS w.r.t.~arbitrary switching, or w.r.t.~$\cS_{adw}^\infty(\tau)$ for any $\tau \in \R$ (recall Lemma~\ref{lemma:ImpyArbitrayStability}).

\section{Lyapunov (Converse) Result For $\GAS$}\label{sec:MainResult}
In this section we provide a self-contained review of classical multiple Lyapunov conditions for average dwell-time stability, introduced in~\cite{HesMor99} (see also \cite{Morse,Lib03}) and further developed/extended in~\cite{LiuTanwLib}.
After some preliminaries discussion, we provide a converse Lyapunov result inspired by the ideas behind the proof of Proposition~\ref{prop:ConverseDwellTime},  formally proving the equivalence of such conditions with the property introduced in Definition~\ref{defn:strongNotion}.

First of all we recall the main Lyapunov sufficient conditions in the  non-linear case.

\begin{prop}\label{prop:AneelConditions}{\emph{(Theorem 1 in~\cite{LiuTanwLib})}}
Consider $\cF=\{f_1,\dots,f_{\rm m}\}\subset \sol(\R^n)$.
Suppose that there exist $\alpha_1,\alpha_2\in \cK_\infty$, $\rho\in \cK$, $\chi\in \cK_\infty$, and $V_1,\dots, V_{\rm m}\in  \cLL_0(\R^n,\R)$ such that
\begin{subequations}\label{eq:AneelCOnditions}
\begin{equation}\label{eq:Sandwich-NonLinear}
\alpha_1(|x|)\leq V_i(x)\leq \alpha_2(|x|),\;\;\;\forall \,i\in \cI,\;\forall x\in \R^n,
\end{equation}
\begin{equation}\label{eq:Decreae-NonLinear}
D^+_{f_i} V_i(x) \leq -\rho(V_i(x)),\;\;\,\;\;\;\;\;\forall \,i\in \cI,\;\forall x\in \R^n,
\end{equation}
\begin{equation}\label{eq:JumpNonlinear}
V_i(x)\leq \chi(V_j(x)),\;\;\;\;\;\;\;\forall\, (i,j)\in \cI^2,\;\forall\,x\in \R^n.
\end{equation}
\end{subequations}
Moreover, consider $\varepsilon\geq 0$ and the function
\[
\Psi_\varepsilon(t):=\min_{s\in [0,t]}\{\rho(s)+\varepsilon(t-s)\}\leq \min\{\rho(t),\varepsilon t\},\;\;\forall \,t\in \R_+.
\]
Suppose that, for a given $\varepsilon>0$, it holds that  
\begin{equation}\label{eq:SmallGainNon_Linear}
 \tau^\star :=\sup_{s>0}\int_{s}^{\chi(s)}\frac{1}{\Psi_\varepsilon(r)}\,dr<+\infty.
\end{equation}
Then, for any $\tau>\tau^\star$, system~\eqref{eq:SwitchedSystem} is $\GAS$.
\end{prop}
We refer to~\cite[Theorem 1]{LiuTanwLib} for a direct proof of Proposition~\ref{prop:AneelConditions}. It must be noted that the notion of $\GAS$ is not explicitly introduced in~\cite{LiuTanwLib}. However, by following the proof of \cite[Theorem 1]{LiuTanwLib}, one can obtain the expression~\eqref{eq:NonlinearRelaxation}, which in particular proves UGAS for $\cS_{adw}(\tau,N_0)$ for all $N_0\in \N$ that were studied in \cite{LiuTanwLib}, as stated in Lemma~\ref{Lemma:StrongConsequences}.
Here we provide a novel proof of Proposition~\ref{prop:AneelConditions}, based on the following statement which will be independently used in what follows.

\begin{lemma}\label{Lemma:LinearizingTheFunctions}
Consider $\cF=\{f_1,\dots,f_{\rm m}\}\subset \sol(\R^n)$. Suppose that there exist $\alpha_1,\alpha_2\in \cK_\infty$, $\rho\in \cK$, $\chi\in \cK_\infty$, $\varepsilon>0$ and $V_1,\dots, V_{\rm m}\in  \cLL_0(\R^n,\R)$ such that conditions~\eqref{eq:AneelCOnditions}-\eqref{eq:SmallGainNon_Linear} are satisfied; then, for any $\tau>\tau^\star$, there exist $W_1,\dots W_{\rm m}\in  \cLL_0(\R^n,\R)$, 
$\wt \alpha_1,\wt \alpha_2\in \cK_\infty$ and $\alpha\geq 0$,  such that
\begin{subequations}\label{eq:AneelLinearized}
\begin{equation}\label{eq:Sandwich-NonLinearBB}
\wt \alpha_1(|x|)\leq W_i(x)\leq \wt \alpha_2(|x|),\;\;\;\;\forall \,i\in \cI,\;\forall x\in \R^n,
\end{equation}
\begin{equation}\label{eq:Decreae-NonLinearBB}
D^+_{f_i} W_i(x) \leq -(1+\alpha)\, W_i(x)\;\;\;\;\forall \,i\in \cI,\;\forall x\in \R^n,
\end{equation}
\begin{equation}\label{eq:JumpConditionNOnlinearBB}
W_i(x)\leq e^{\alpha \tau}W_j(x),\;\;\;\;\;\,\;\;\forall\, (i,j)\in \cI^2,\;\forall\,x\in \R^n.
\end{equation}
\end{subequations}
Conversely, if there exist $W_1,\dots, W_{\rm m}\in  \cLL_0(\R^n,\R)$, 
$\wt \alpha_1,\wt \alpha_2\in \cK_\infty$, $\alpha\geq 0$ and $\tau>0$ satisfying~\eqref{eq:AneelLinearized} then there exist $\tau^\star<\tau$ and $\alpha_1,\alpha_2\in \cK_\infty$, $\rho\in \cK$, $\chi\in \cK_\infty$, $\varepsilon>0$ and $V_1,\dots, V_{\rm m}\in  \cLL_0(\R^n,\R)$ such that conditions~\eqref{eq:AneelCOnditions}-\eqref{eq:SmallGainNon_Linear} are satisfied.
\end{lemma}
\begin{rem}[Common Lypunov function] We note that the case $\alpha=0$ in conditions~\eqref{eq:AneelLinearized} corresponds to the fact that any $W_i$ is a \emph{common Lyapunov function} for the considered switched system. Indeed, inequality~\eqref{eq:JumpConditionNOnlinearBB} in this case reads $W_i(x)\leq W_j(x)$, for all $(i,j)\in \cI^2,\;\forall\,x\in \R^n$, thus implying that $W_i\equiv W_j$ for all $(i,j)\in \cI^2$, i.e., all the functions coincide with a unique $W\in \cLL_0(\R^n,\R)$. Then inequality~\eqref{eq:Sandwich-NonLinearBB} will provide the classical positive definetenss property for $W$, while~\eqref{eq:Decreae-NonLinearBB} now reads $D^+_{f_i} W(x) \leq - W(x)$,  for all $i\in \cI$ and all $x\in \R^n$.  
In other words, $W$ is a common Lyapunov function for system~\eqref{eq:SwitchedSystem} (as introduced in~\cite{MolPya89,DayaMar99,Mancilla00}, for example), which is thus UGAS~with respect to $\cS$. This observation is in line with the discussion provided in Remark~\ref{rem:ArbitraryStability}.
\end{rem}

\begin{proof}[Proof of Lemma~\ref{Lemma:LinearizingTheFunctions}]
Let us first suppose that $\alpha_1,\alpha_2\in \cK_\infty$, $\rho\in \cK$, $\chi\in \cK_\infty$, and $V_1,\dots, V_{\rm m}\in \cC^1(\R^n,\R)$ satisfy conditions~\eqref{eq:AneelCOnditions}, for some $\varepsilon>0$ and $\tau^\star\geq 0$ such that \eqref{eq:SmallGainNon_Linear} holds. Consider any $\tau>\tau^\star$ and take $\alpha\geq 0$ such that 
\[
\frac{\tau^\star}{\tau}=\frac{\alpha}{1+\alpha}.
\]
Let us consider the function $\gamma:\R_+\to \R_+$ defined by
\[
\gamma(s):=\begin{cases}
0,\;\;\;&\text{if } s= 0,\\
\text{exp}(\int_1^s \frac{1+\alpha}{\Psi_\varepsilon(r)}\,dr),\;\;\;\;&\text{if } s\neq 0.
\end{cases}
\]
Using the fact that $\Psi_\varepsilon$ is globally Lipschitz (with $\varepsilon$ as Lipschitz constant), in \cite{LiuTanwLib} (see also \cite[Lemmas 11 \& 12]{PralyWnag96}) it is proved that $\gamma\in \cK_\infty \cap \cC^1(\R_+,\R)$. In particular, we have
\[
\gamma'(s)=\frac{1+\alpha}{\Psi_\varepsilon(s)}\gamma(s),\;\;\;\;\forall\,s>0.
\]
We now define $W_i(x):=\gamma(V_i(x))$. It is clear that choosing $\wt \alpha_1=\gamma\circ \alpha_1$ and $\wt \alpha_2=\gamma \circ\alpha_2$, condition~\eqref{eq:Sandwich-NonLinearBB} holds. Consider now any $x\neq 0$, by definition, we have
\[
\begin{aligned}
D^+_{f_i} W_i(x)&=\limsup_{h\to 0^+} \frac{\gamma(V_i(\Phi_i(h,x)))-\gamma( V_i(x))}{h}=\lim_{h\to 0^+}\frac{\gamma(V_i(\Phi_i(h,x)))-\gamma(V_i(x))}{V_i(\Phi_i(h,x))-V_i(x)}\limsup_{h\to 0^+}\frac{V_i(\Phi_i(h,x))-V_i(x)}{h},
\end{aligned}
\]
where the last equality is justified since: \begin{itemize}[leftmargin=*]
\item $D^+_{f_i}V_i(x)<0$  implies that there exists $\bar h>0$ such that $V_i(\Phi_i(h,x))-V_i(x)< 0$ for all $h\in (0,\bar h)$;
\item $\lim_{h\to 0^+}\frac{\gamma(V_i(\Phi_i(h,x)))-\gamma(V_i(x))}{V_i(\Phi_i(h,x))-V_i(x)}$ exists, and it is equal to $\gamma'(V_i(x))\geq 0$, since $\lim_{h\to 0^+}V_i(\Phi_i(h,x))=V_i(x)$.
\end{itemize}
We thus have
\[
\begin{aligned}
D^+_{f_i} W_i(x)&=\lim_{y\to V_i(x)}\frac{\gamma(y)-\gamma(V_i(x))}{y-V_i(x)}\limsup_{h\to 0^+}\frac{V_i(\Phi_i(h,x))-V_i(x)}{h}=\gamma'(V_i(x))D^+_{f_i}V_i(x).
\end{aligned}
\]
Then, we obtain
\[
\begin{aligned}
D^+_{f_i} W_i(x)&=\gamma'(V_i(x))D^+_{f_i}V_i(x)\leq  -\gamma'(V_i(x))\rho(V_i(x))\\&=-\frac{1+\alpha}{\Psi_\varepsilon(V_i(x))}\rho(V_i(x))\gamma(V_i(x))\leq -(1+\alpha)\gamma(V_i(x))=-(1+\alpha)W_i(x),
\end{aligned}
\]
which proves~\eqref{eq:Decreae-NonLinearBB} for any $x\neq 0$. Since~\eqref{eq:Decreae-NonLinearBB} is trivial in the case $x=0$, it remains to prove that~\eqref{eq:JumpConditionNOnlinearBB} holds.
\\
Let us consider the $\cK_\infty$ function $\wt \chi:=\gamma \circ \chi \circ \gamma^{-1}$, by~\eqref{eq:JumpNonlinear} we have
\[
W_i(x)\leq \wt \chi(W_j(x)),\;\;\;\forall (i,j)\in \cI^2,\;\forall\,x\in\R^n.
\]
We then obtain
\[
\begin{aligned}
\frac{1}{(1+\alpha)}\sup_{s>0}\left [\log\left(\frac{\wt \chi(s)}{s}\right)\right] &=\frac{1}{(1+\alpha)}\sup_{s>0}\int_s^{\wt \chi(s)} \frac{1}{r}\;dr
=\sup_{u>0}\int_{\gamma(u)}^{(\gamma\circ \chi)(u)}\frac{1}{(1+\alpha) r}\;dr\\&= \sup_{u>0}\int_{u}^{\chi(u)}\frac{\gamma'(r)}{(1+\alpha) \gamma(r)}\;dr= \sup_{u>0}\int_{u}^{\chi(u)}\frac{1}{\Psi_\varepsilon(r)}\;dr=\tau^\star.
\end{aligned}
\]
This implies that 
\[
\wt \chi(s)\leq e^{\tau^\star (1+\alpha)}s=e^{\tau\alpha}s  ,\;\;\;\forall\,s\in \R_+,
\]
which proves~\eqref{eq:JumpConditionNOnlinearBB}, concluding the first direction of the proof.
\\
The other implication is straightforward. Consider any $\tau>0$ and suppose that the functions $W_1, \dots, W_{\rm m}$ satisfying~\eqref{eq:AneelLinearized} exist (for some $\wt \alpha_1,\wt \alpha_2\in \cK_\infty$ $\alpha\geq 0$ and $\tau>0$). If $\alpha=0$, then we can choose $\chi \equiv \rho\equiv \text{Id}$ and the conditions in~\eqref{eq:AneelCOnditions}-\eqref{eq:SmallGainNon_Linear} trivially hold with $\tau^ \star=0$ and $V_i=W_i$ for any $i\in \cI$. If $\alpha>0$, consider $ \tau^\star<\tau$ such that  $\frac{\tau^\star}{\tau}=\frac{\alpha}{1+\alpha}$, then the functions  $V_i=W_i$ for any $i\in \cI$ satisfy~\eqref{eq:AneelCOnditions}, by choosing $\rho(s):=(1+\alpha)s$, $\chi(s)=e^{\alpha \tau}s$, $\varepsilon=\alpha+1$. Moreover, choosing $\Psi_{\alpha+1}(s)=(1+\alpha)s$, it can be seen that condition~\eqref{eq:SmallGainNon_Linear} holds, i.e.,
\[
\begin{aligned}
\sup_{s>0}\int_{s}^{\chi(s)}\frac{1}{\Psi_{\alpha+1}(r)}&=\sup_{s>0}\int_{s}^{e^{\alpha \tau}s}\frac{1}{(1+\alpha)r}\,dr=\frac{1}{1+\alpha}\sup_{s>0}\int_{s}^{e^{\alpha \tau}s}\frac{1}{r}\,dr\\&=\frac{1}{1+\alpha}\sup_{s>0}\left (\,\ln(e^{\alpha\tau}s)-\ln(s)\right)=\frac{\alpha \tau}{1+\alpha}=\tau^\star,
\end{aligned}
\]
concluding the proof.
\end{proof}

We now prove Proposition~\ref{prop:AneelConditions} using Lemma~\ref{Lemma:LinearizingTheFunctions}.
\begin{proof}[Proof of Proposition~\ref{prop:AneelConditions}]
Using Lemma~\ref{Lemma:LinearizingTheFunctions}, let us consider $W_1,\dots W_{\rm m}\in \cLL_0(\R^n,\R)$, 
$\wt \alpha_1,\wt \alpha_2\in \cK_\infty$ and $\alpha\geq 0$ satisfying~\eqref{eq:Sandwich-NonLinearBB},~\eqref{eq:Decreae-NonLinearBB} and~\eqref{eq:JumpConditionNOnlinearBB}. Consider any $\sigma\in \cS$, and any $t\in \R_+$ and suppose that $t\in [t_k^\sigma, t_{k+1}^\sigma)$, for some $t_k^\sigma\geq 0$. Using~the flow condition \eqref{eq:Decreae-NonLinearBB} together with the comparison principle applied to the continuous function $W_i(\Phi_i(\,\cdot\,,x)):\R_+\to \R^n$ (see for example~\cite[Lemma 3.4]{khalil2002nonlinear}) and applying the jump condition \eqref{eq:JumpConditionNOnlinearBB} recursively, we obtain
\begin{equation}\label{eq:MotivationStrongNotion}
\begin{aligned}
W_{\sigma(t)}(\Phi_\sigma(t,x))&\leq  e^{-(1+\alpha)(t-t^\sigma_k)}W_{\sigma(t^\sigma_k)}(\Phi_\sigma(t^\sigma_k,x))\leq e^{\alpha \tau}e^{-(1+\alpha)(t-t^\sigma_{k-1})}W_{\sigma(t^\sigma_{k-1})}(\Phi_\sigma(t^\sigma_{k-1},x))\leq \dots \\&\leq e^{\alpha \tau N_\sigma(0,t)} e^{-(1+\alpha) t}W_{\sigma(0)}(x).
\end{aligned}
\end{equation}
From \eqref{eq:Sandwich-NonLinearBB}, it follows that
\[
|\Phi_\sigma(t,x)|\leq \wt \alpha_1^{-1}( e^{\alpha \tau N_\sigma(0,t)}e^{-(1+\alpha)t}\wt \alpha_2(|x|))\;\;\forall \sigma\in \cS,\;\;\forall x\in \R^n,\;\;\forall t\in \R_+.
\]
By letting $\eta_1:=\wt \alpha_1^{-1}$ and $\eta_2:=\wt \alpha_2$, we have thus proved~\eqref{eq:NonlinearRelaxation} and this concludes the proof of Proposition~\ref{prop:AneelConditions}.
\end{proof}

In the following statement, we present our main result establishing the equivalence of $\GAS$ of~\eqref{eq:SwitchedSystem} with the existence of functions satisfying~\eqref{eq:AneelCOnditions}-\eqref{eq:SmallGainNon_Linear} (via the intermediate step provided by Lemma~\ref{Lemma:LinearizingTheFunctions}).
\begin{thm}\label{thm:Converse} Consider $\cF=\{f_1,\dots, f_{\rm m}\}\subset\sol(\R^n)$ and $\tau>0$. System~\eqref{eq:SwitchedSystem} is $\GAS$ if and only if there exist $W_1,\dots, W_{\rm m}\in\cLL_0(\R^n,\R)$, $\wt \alpha_1,\wt\alpha_2\in \cK_\infty$, $\alpha\geq 0$  satisfying~\eqref{eq:AneelLinearized}.
\end{thm}
\begin{proof}
The proof of sufficiency follows from~Proposition~\ref{prop:AneelConditions} and~Lemma~\ref{Lemma:LinearizingTheFunctions}.
\\
\emph{(Necessity):} 
Suppose that the system is $\GAS$, and let us fix $\varepsilon>0$. Recalling Definition~\ref{defn:strongNotion} and Lemma~\ref{lemma:SpeedTransformation}, there exist $\eta_1,  \eta_2\in \cK_\infty$ and $\alpha\geq 0$ such that 
\[
|\Phi_\sigma(t,x)|\leq \eta_1(e^{\alpha\tau N_\sigma(0,t)}e^{-(1+\varepsilon+ \alpha) t} \eta_2(|x|)),\;\;\;\;\forall \sigma \in \cS,\;\forall x\in \R^n,\;\forall t\in \R_+.
\]
Without loss of generality, we can assume that $\eta_1\in \cK_\infty\cap \cC^1(\R_+\setminus\{0\},\R)$ and $\eta_1'(s)>0$ for all $s >0$ (see \cite[Lemma 1]{Kellett2014}).
Let us define $\gamma :=1+\alpha$, we thus have
\begin{equation}\label{eq:BoundSolutionProof}
\frac{e^{\gamma t}}{e^{\alpha \tau N_\sigma(0,t)}}\alpha_1(|\Phi_\sigma(t,x)|)\leq \eta_2(|x|)e^{-\varepsilon t}\leq \eta_2(|x|), \;\;\;\;\forall \sigma \in \cS,\;\forall x\in \R^n,\;\forall t\in \R_+,
\end{equation}
with $\alpha_1=\eta_1^{-1}$.
We define, for every $i\in \cI$, the set
\[
\cS_i:=\{\sigma \in \cS\;\vert\;\sigma(0)=i\},
\]
i.e., the signals that ``start'' with the value equal to $i\in \cI$. 
 We then define, for all $i\in \cI$, the function $W_i:\R^n\to \R$ by
\begin{equation}\label{eq:DefinitionV_i}
W_i(x):=\sup_{\sigma\in \cS_i}\sup_{s\geq 0} \frac{e^{\gamma s}}{e^{\alpha \tau N_\sigma(0,s)}}\alpha_1(|\Phi_\sigma(s,x)|).
\end{equation}
For every $i\in \cI$, let us denote by $\nu_i\in \cS_i$ the constant signal, that is $\nu_i(t)=i$, for all $t\geq 0$.
We have
\[
W_i(x)=\sup_{\sigma\in \cS_i}\sup_{s\geq 0} \frac{e^{\gamma s}}{e^{\alpha \tau N_\sigma(0,s)}}\alpha_1(|\Phi_\sigma(s,x)|)\geq \sup_{s\geq 0} \frac{e^{\gamma s}}{e^{\alpha \tau N_{\nu_i}(0,s)}}\alpha_1(|\Phi_{\nu_i}(s,x)|)\geq \frac{1}{e^{\alpha \tau}}\alpha_1(|x|),
\]
where in the last inequality, we have chosen $s=0$ as sample time.
On the other hand, equation~\eqref{eq:BoundSolutionProof} implies
\[
W_i(x)\leq \eta_2(|x|),\;\;\forall \,i\in \cI,\;\forall x\in \R^n,
\]
and we have thus proved that the $W_i$, for each $i \in \cI$, satisfies the inequality~\eqref{eq:Sandwich-NonLinearBB} with $\wt \alpha_1=\frac{1}{e^{\alpha\tau}}\eta_1^{-1}$ and $\wt \alpha_2 = \eta_2$.

Since from now on we use concatenation arguments, we refer to Appendix~\ref{app:Concat} for the main definitions and results.
We now prove inequality~\eqref{eq:Decreae-NonLinearBB}.
Considering any $i\in \cI$, any $x\in \R^n$ and any $t\geq 0$, we have
\[
\begin{aligned}
W_i(\Phi_i(t,x))=\sup_{\sigma\in \cS_i}\sup_{s\geq 0} \frac{e^{\gamma s}}{e^{\alpha \tau N_\sigma(0,s)}}\alpha_1(|\Phi_\sigma(s,\Phi_i(t,x))|)=\sup_{\sigma\in \cS_i}\sup_{s\geq 0} \frac{e^{\gamma s}}{e^{\alpha \tau N_\sigma(0,s)}}\alpha_1(|\Phi_{\nu_i\diamond_t \sigma}(s+t,x)|).
\end{aligned}
\] 
Recalling Lemma~\ref{lemma:Concatenation} in Appendix~\ref{app:Concat}, for any $\sigma\in \cS_i$ and for any $t\in \R_+$, we have  $\nu_i\diamond_t \sigma\in \cS_i$ and $N_{\nu_i\diamond_t \sigma}(0,s+t)= N_{\sigma}(0,s)$ for any $s\geq 0$. We thus proceed as follows:
\[
\begin{aligned}
W_i(\Phi_i(t,x))&= \sup_{\sigma\in \cS_i}\sup_{s\geq 0} \frac{e^{\gamma s}}{e^{\alpha \tau N_{\nu_i\diamond_t \sigma}(0,s+t)}}\alpha_1(|\Phi_{\nu_i\diamond_t \sigma}(s+t,x)|)\\&\leq \sup_{\psi\in \cS_i}\sup_{r \geq t} \frac{e^{\gamma(r-t)}}{e^{\alpha \tau N_\gamma(0,r)}}\alpha_1(|\Phi_{\psi}(r,x)|)\leq e^{-\gamma t}\sup_{\psi\in \cS_i}\sup_{r\geq 0} \frac{e^{\gamma r}}{e^{\alpha \tau N_\gamma(0,r)}}\alpha_1(|\Phi_{\psi}(r,x)|)=e^{-\gamma t} W_i(x).
\end{aligned}
\]
We have thus proved that $W_i(\Phi_i(t,x))\leq e^{-\gamma t} W_i(x)$, for all $i\in \cI$, for all $x\in \R^n$ and all $t\in \R_+$. If the functions $W_i$ are locally Lipschitz on $\R^n \setminus \{0\}$ (as we will prove in what follows), this implies~\eqref{eq:Decreae-NonLinearBB}. Indeed, it holds that 
\[
D_{f_i}^+W_i(x)=\limsup_{h\to 0^+}\frac{W_i(\Phi_i(h,x))-W_i(x)}{h}\leq \limsup_{h\to 0^+}\frac{e^{-\gamma h}W_i(x)-W_i(x)}{h}=W_i(x)\lim_{h\to 0}\frac{e^{-\gamma h}-1}{h}=-\gamma W_i(x),
\]
for any $x\in \R^n$.

We now prove~\eqref{eq:JumpConditionNOnlinearBB}. Since $e^{\alpha\tau}\geq 1$, the case $i=j$ is trivial, we thus suppose that $i\neq j$. The core idea is to concatenate the constant signal $\nu_i$ with an arbitrary $\psi\in \cS_j$, on an initial interval of arbitrarily short length. The obtained signal will be in $\cS_i$, but the corresponding solutions will be close, in a sense we clarify, to the ones corresponding to  $\psi$, if the initial interval is small.
More formally, for every $x\in \R^n$ and for an arbitrary $\delta>0$, we have
\[
\begin{aligned}
  W_i(x)&=\sup_{\sigma\in \cS_i}\sup_{s\geq 0} \frac{e^{\gamma s}}{e^{\alpha \tau N_\sigma(0,s)}}\alpha_1(|\Phi_\sigma(s,x)|)\geq \sup_{\psi\in \cS_j}\sup_{s\geq 0} \frac{e^{\gamma s}}{e^{\alpha \tau N_{\nu_i\diamond_\delta \psi}(0,s)}}\alpha_1(|\Phi_{\nu_i\diamond_\delta \psi}(s,x)|),
  \\&\geq \sup_{\psi\in \cS_j}\sup_{r\geq 0} \frac{e^{\gamma (r+\delta)}}{e^{\alpha \tau N_{\nu_i\diamond_\delta \psi}(0,r+\delta)}}\alpha_1(|\Phi_\psi(r,\Phi_i(\delta,x)|),
\end{aligned}
\]
where, in the last step, we restricted the supremum over $s\geq \delta$ and we performed the change of time-variable $s=\delta+r$.
Since in Lemma~\ref{lemma:Concatenation} of Appendix~\ref{app:Concat}, we proved that $N_{\nu_i\diamond_\delta \psi}(0,r+\delta)=N_\psi(0,r)+N_{\nu_i}(0,\delta)= N_\psi(0,r)+1$ for every $\psi\in \cS_j$ and every $r\geq \delta$, we have 
\[
\begin{aligned}
   W_i(x)&\geq \sup_{\psi\in \cS_j}\sup_{r\geq 0} \frac{e^{\gamma (r+\delta)}}{e^{\alpha \tau (N_\psi(0,r)+1)}}\alpha_1(|\Phi_{\psi}(r,\Phi_i(\delta,x))|)\\
& \geq \frac{e^{\gamma \delta}}{e^{\alpha \tau}}\sup_{\psi\in \cS_j}\sup_{r\geq 0} \frac{e^{\gamma  r}}{e^{\alpha \tau N_\psi(0,r)}}\alpha_1(|\Phi_{\psi}(r,\Phi_i(\delta,x))|)= \frac{e^{\gamma \delta}}{e^{\alpha \tau}}W_j(\Phi_i(\delta, x)).
\end{aligned}
\]
Since the previous inequality holds for any $\delta>0$ and $\lim_{\delta\to 0^+}\frac{e^{\gamma \delta}}{e^{\alpha \tau}}W_j(\Phi_i(\delta, x))=\frac{1}{e^{\alpha \tau}}W_j(x)$, the inequality in \eqref{eq:JumpConditionNOnlinearBB} holds.
\\
It remains to prove that $W_i\in \cLL_0(\R^n,\R)$, for all $i\in \cI$.
This can be done substantially following the arguments in~\cite[Section 5]{TeelPraly2000}. The proof is technical and thus only sketched in Lemma~\ref{lemma:LocLipschitz} in Appendix, to avoid breaking the flow of the presentation.
\end{proof}

\begin{rem}
The contribution of Theorem~\ref{thm:Converse} is somehow bi-fold. From one side it provides a multiple Lyapunov functions characterization of $\GAS$ property introduced in Definition~\ref{defn:strongNotion}. On the other side, it also formally describes the conservatism of Lyapunov sufficient conditions as in~Proposition~\ref{prop:AneelConditions}, closing the gap of a long and fruitful story of Lyapunov conditions for average dwell-time stability that can be traced back to~\cite{HesMor99}. This equivalence has been proved in a generic \emph{non-linear} subsystems setting. 
It is well-known that in such general non-linear case, non-global/practical stability phenomena can arise in a switching systems context, see for example~\cite{DelROsTan22,LiuTanwLib} and references therein. For simplicity and conciseness, we did not provide in this paper the local/practical versions of Theorem~\ref{thm:Converse} and this line of research is open for future investigation.
\end{rem}

\section{Linear Subsystems Case}\label{sec:Linear}
In this section, we specialize Theorem~\ref{thm:Converse} to the linear case. Using linearity and following the results in~\cite{Lib03,AngeliNote} and~\cite[Chapter 5]{BacRosier}, it can be proved that in this context, given any set $\wt \cS\subseteq \cS$ closed under time right-shifting\footnote{$\wt\cS\subseteq \cS$ is \emph{closed under time right-shifting} if $\sigma\in \wt \cS$ $\Leftrightarrow$ $\sigma(\cdot+t)\in \wt \cS$,  $\forall t\in \R_+$.}, (UGAS) w.r.t.~$\wt \cS$ imply ($\text{UGES}_\rho$) (for a certain $\rho>0$) w.r.t.~$\wt \cS$. Similarly, it can be proved that, for linear switched systems, $\GAS$ implies $\ES$, for a certain $\rho>0$. For this reason, in this section we focus on the $\ES$ notion provided in Definition~\ref{defn:strongNotion}.

As for the characterization of $\text{UGES}_\rho$ w.r.t.~$\cS_\dw(\tau)$ provided in~\cite{Wirth2005} (and reported in Proposition~\ref{prop:ConverseDwellTime}), we show that $\ES$ of~\eqref{eq:LinearSwitchedSystem} is equivalent to the existence of multiple Lyapunov \emph{norms}, with properties similar to the ones provided in Theorem~\ref{thm:Converse} for the non-linear case.
\begin{cor}\label{cor:ConverseLinear}
Given $\cA=\{A_1,\dots A_{\rm m}\}\subset \R^{n\times n}$, $\tau>0$ and $\rho>0$, system~\eqref{eq:LinearSwitchedSystem} is $\ES$ if and only if there exist norms $v_1,\dots, v_{\rm m}:\R^n\to \R$ and $\alpha\geq 0$ such that
\begin{subequations} \label{eq:daad}
\begin{align}
D^+_{A_i}v_i(x)&\leq -(\rho+\alpha) v_i(x), \;\;\;\forall i\in \cI, \;\forall x\in \R^n,\\
v_i(x)&\leq e^{\alpha\tau} v_j(x),\;\;\;\forall x\in \R^n, \;\forall \;(i,j)\in \cI^2.
\end{align}
\end{subequations}
\end{cor}
\begin{proof}
The proof basically follows from Theorem~\ref{thm:Converse}. The sufficiency is trivial. For the the necessity, recalling Definition~\ref{defn:strongNotion}, there exist $\alpha\geq 0$ and $M\geq 1$ such that 
\begin{equation}\label{eq:DecraseToolNorms}
|\Phi_\sigma(t,x)|\leq M e^{\alpha \tau N_\sigma(0,t)}e^{-(\rho+\alpha)t}|x|,\;\;\;\forall\,\sigma\in \cS,\;\forall\,x\in \R^n,\,\forall\,t\in \R_+.
\end{equation}
 Recalling that $\cS_i=\{\sigma \in \cS\;\vert\;\sigma(0)=i\}$, we define, for every $i\in \cI$,
\[
  v_i(x):=\sup_{\sigma\in \cS_i}\sup_{s\geq 0}\frac{e^{(\rho+\alpha)t}}{e^{\alpha \tau N_\sigma(0,s)}}|\Phi_\sigma(s,x)|.
\]
By~\eqref{eq:DecraseToolNorms} it holds that
\[
  \frac{1}{e^{\alpha\tau}}|x|\leq v_i(x)\leq M|x|,\;\;\;\forall \;i\in \cI,\;\forall x\in \R^n.
\]
The inequalities in~\eqref{eq:daad} can now be proved with the arguments presented in proof of Theorem~\ref{thm:Converse}. 
It remains to show that $v_i$ are norms. For this, we recall that, by linearity, the flows maps of~\eqref{eq:LinearSwitchedSystem} are linear, i.e.,
\[
  x\mapsto \Phi_\sigma(t,x)\;\text{ is linear},\;\;\forall \,\sigma\in \cS,\,\forall \,t\in \R_+. 
\]
For this reason, for any $\lambda\in \R$, we have 
\[
 v_i(\lambda x)=\sup_{\sigma\in \cS_i}\sup_{s\geq 0}\frac{e^{-(\rho+\alpha)t}}{e^{\alpha\tau N_\sigma(0,s)}}|\Phi_\sigma(s,\lambda x)|=|\lambda| \sup_{\sigma\in \cS_i}\sup_{s\geq 0}\frac{e^{-(\rho+\alpha)t}}{e^{\alpha\tau N_\sigma(0,s)}}|\Phi_\sigma(s, x)|=|\lambda| v_i(x),
\]
while for the triangular inequality, considering $x_1,x_2\in \R^n$ we have
\[
\begin{aligned}
  v_i(x_1+x_2)&=\sup_{\sigma\in \cS_i}\sup_{s\geq 0}\frac{e^{-(\rho+\alpha)t}}{e^{\alpha \tau N_\sigma(0,s)}}|\Phi_\sigma(s, x_1+x_2)|\\
& \leq \sup_{\sigma\in \cS_i}\sup_{s\geq 0}\frac{e^{-(\rho+\alpha)t}}{e^{\alpha \tau N_\sigma(0,s)}}|\Phi_\sigma(s, x_1)|+\sup_{\sigma\in \cS_i}\sup_{s\geq 0}\frac{e^{-(\rho+\alpha)t}}{e^{\alpha \tau N_\sigma(0,s)}}|\Phi_\sigma(s,x_2)|=v_i(x_1)+v_i(x_2),
\end{aligned}
\]
concluding the proof.
\end{proof}

We note here for historical reasons that the conditions of Corollary~\ref{cor:ConverseLinear}, when restricted to \emph{quadratic norms} (and thus loosing the necessity), read as follow:\\
For a fixed $\rho>0$, if there exists $P_1,\dots, P_{\rm m}\succ 0$, $\alpha\geq 0$ and $\nu\geq 1$ such that
\begin{equation}\label{eq:QuadraticProblem}
\begin{aligned}
P_iA_i+A_i^\top P_i\prec -2(\alpha+\rho) P_i,\\
P_i\preceq \nu^2 P_j,
\end{aligned}
\end{equation}
then the system~\eqref{eq:LinearSwitchedSystem} is $\ES$ for $\tau\geq \frac{\ln(\nu)}{\alpha}$.

These inequalities exactly correspond to the conditions provided in the seminal paper~\cite[Theorem 2]{HesMor99} (modulo some changes in the notation).

\begin{rem}[Minimum (average) dwell-time]
Given $A\in \R^{n\times n}$, we introduce the~\emph{spectral abscissa of $A$}, defined as $\lambda(A) : =\max_{i\in\{1,\dots, n\}} \text{Re}(\lambda_i)$, where $\lambda_1,\dots,\lambda_n$ denote the (possibly complex) eigenvalues of $A$. Given an arbitrary norm $\|\cdot\|$ in $\R^n$, and its associated operator norm\footnote{Given a norm $\|\cdot\|:\R^n\to \R$, the \emph{associated operator norm} on $\R^{n\times n}$ is defined by $\|A\|:=\sup_{x\in \R^n\setminus\{0\}}\frac{\|Ax\|}{\|x\|}$, for $A\in \R^{n\times n}$.}, for every $\varepsilon>0$ there exists $M_\varepsilon\geq 1$ such that $\|e^{At}\|\leq M_\varepsilon e^{(\lambda(A)+\varepsilon)t}$, for all $t\in \R_+$. A matrix $A$ is said to be \emph{Hurwitz} if $\lambda(A)<0$. Let us now consider $\cA=\{A_i\}_{i\in \cI}$, and suppose $A_i$ is Hurwitz for all $i\in \cI$, we introduce the notation $\lambda(\cA):=\max_{i\in \cI}\lambda(A_i)<0$. It can be easily proved that, for every $0<\rho<|\lambda(\cA)|$, there exist $\alpha\geq 0$ and $\tau>0$ (large enough) and norms $v_1,\dots, v_{\rm m}:\R^n\to \R$ such that the conditions in Corollary~\ref{cor:ConverseLinear} are satisfied. Thus, the conditions in Corollary~\ref{cor:ConverseLinear}, once we fix a desired and feasible decay rate $\rho<|\lambda(\cA)|$,  can be used to provide a ``safety'' value of $\tau$, for which~\eqref{eq:LinearSwitchedSystem} is $\ES$. If we are able to search candidate Lyapunov norms over the whole class of norms, and to arbitrarily vary the parameter $\alpha\in \R_+$, Corollary~\ref{cor:ConverseLinear} ensures that we will recover the best value for $\tau$, i.e. the so-called \emph{minimal (average) dwell-time}. More formally, given  $\cA=\{A_1,\dots, A_{\rm m}\}$ and $\rho<|\lambda(\cA)|$, we define
\[
\overline \tau_{min}(\cA,\rho):=\inf\left\{\tau\geq 0\;\vert\;\text{system~\eqref{eq:LinearSwitchedSystem} is $\ES$} \right\}, 
\]
which, by the discussion above, is bounded  for any $\rho<\lambda(\cA)$.
The problem of computing/estimating $\overline \tau_{min}(\cA,\rho)$ is challenging, and it can be considered related to the celebrated open question proposed in~\cite{Hespanha04} (see also~\cite{ChiMas17} for a recent overview).

Rephrasing Corollary~\ref{cor:ConverseLinear} in the light of above discussions, and using the notation $\cN_n$ for the set of norms over $\R^n$, we arrive at
\begin{equation}\label{eq:OptiProble}
\begin{aligned}
\overline \tau_{min}(\cA,\rho)&=\inf_{\tau,\alpha,v_1,\dots v_{\rm m}} \quad  \tau\\
&\textrm{ s.t. } \;\tau\geq 0,\;\alpha\geq 0,\;v_1,\dots, v_{\rm m}\in \cN_n,  \\ & \;\;\;\;\;\;\text{ and conditions~\eqref{eq:daad} are satisfied.}
\end{aligned}
\end{equation}
The optimization problem~\eqref{eq:OptiProble} characterizes $\overline \tau_{min}(\cA,\rho)$, but has the following drawbacks: it requires to search over the whole class of norms and it has a possibly unbounded variable $\alpha\geq 0$.
The first issue can be handled by relaxation, restricting the functional space in which the optimization is performed, for instance, considering quadratic norms (as in~\eqref{eq:QuadraticProblem}), SOS polynomials, polyhedral functions (as in~\cite{HafTan23}), etc. These relaxations  ``convexify'' the problem (for a fixed $\alpha\geq 0$), paying the price of loosing the optimality, and thus only providing upper bounds for $\overline \tau_{min}(\cA,\rho)$.  The second issue (the dependence from $\alpha$) can be handled by line search, see for example~\cite{HafTan23} for the formal discussion.
\end{rem} 
\begin{figure}[t!] 
\centering
\includegraphics[width=.68\linewidth, height =.46\linewidth]{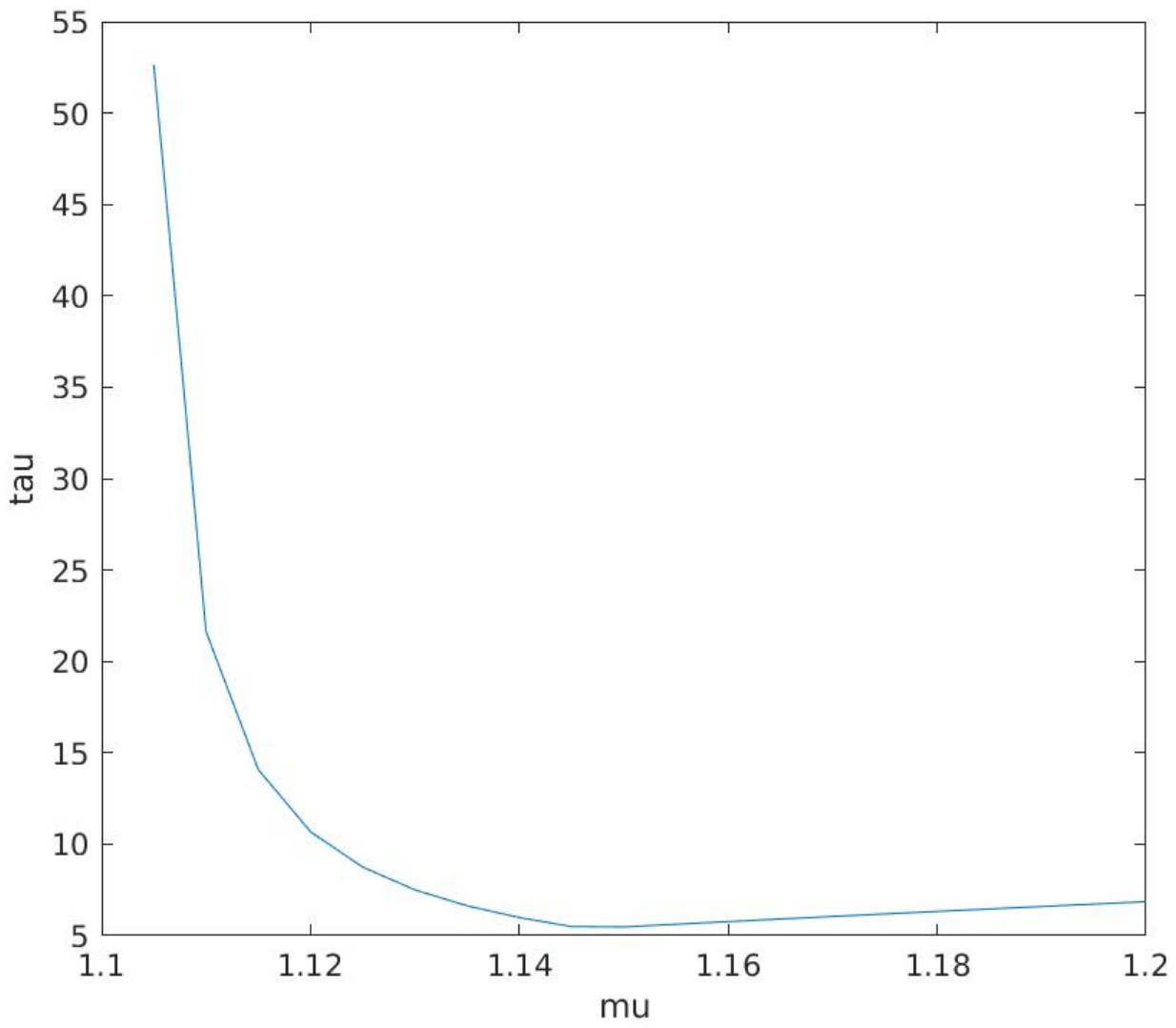}
    \caption{The graph of the minimum $\tau$ in Example~1 obtained by the techniques of \cite{HafTan23}, parametrized by the value of $\mu=e^{\alpha\tau}$.}\label{fig:Parameters}
\end{figure}
\setcounter{myexample}{0}{
\begin{example}{\emph{(Continued)}}
We consider again the planar switched linear system considered in~Example~\ref{ex:ExampleInstability}. We have already proved that such system is UGES w.r.t.~$\cS_{\dw}( \tau_\star)$ with $\tau_\star=2.1$ and that is unstable for $\cS_{\ad}(\tau_\star,2)$.
Thus, it is certainly not~$\tau_\star\text{-UGEB}_\rho$ or, in other words, we have  $\overline \tau_{min}(\cA,\rho)>\tau_\star$,  for any $\rho\in (0,\, |\lambda(\cA)|)=(0, \,0.03)$. We now use the conditions of Corollary~\ref{cor:ConverseLinear} to provide upper bounds on $\overline \tau_{min}(\cA,\rho)$ for a small value of $\rho > 0$ (chosen according to machine precision). In particular, we use the numerical scheme presented in~\cite{HafTan23} in which the research in~\eqref{eq:OptiProble} is restricted over a class of homogeneous functions with polyhedral level sets, and the variation of the parameter $\alpha$ is handled by line-search. The results are illustrated  in Figure~\ref{fig:Parameters}. The best upper bound for $\tau$ is obtained considering $e^{\alpha \tau}=\mu= 1.15$ and it is equal to $\overline \tau_{num}=5.4744$. This proves that, for any $\tau>\overline\tau_{num}$ the system is $\ES$. Note that, as somehow predicted by the discussion provided in Subsection~\ref{subsec:Counter}, the computed upper bound $\overline \tau_{num}$ of $\overline \tau_{min}(\cA,\rho)$ is considerably higher ($2.6$ times) than the upper bound for the minimal \emph{dwell-time}, i.e. the minimal $\tau$ for which the system is uniformly exponentially stable w.r.t.~$\cS_{\dw}(\tau)$.   
\end{example}

\section{Conclusions}\label{sec:conclu}
As an open question for future research, we propose the following conjecture, for which the preformed analysis did not allow us to provide a complete answer.

\begin{conjecture}\label{conj:Conje}
Consider any $\tau>0$. System~\eqref{eq:SwitchedSystem} is $\GAS$ \emph{if and only if} it is UGAS w.r.t.~$\cS_{\ad}(\tau,N_0)$ for any $N_0\in \N$.
Given~$\rho>0$, it is $\ES$ \emph{if and only if} it is  $\text{UGES}_\rho$ w.r.t.~$\cS_{\ad}(\tau,N_0)$ for all $N_0\in \N$.
\end{conjecture}
 Conjecture~\ref{conj:Conje} aims to clarify the relations between the $\GAS$ (resp. $\ES$) property introduced in Definition~\ref{defn:strongNotion} for which we are able to provide converse Lyapunov result in Theorem~\ref{thm:Converse}, and the more ``intuitive'' property of being UGAS (resp. $\text{UGES}_\rho$) w.r.t.~$\cS_{\ad}(\tau,N_0)$ for all $N_0\in \N$. The ``only if'' part of the conjecture has already been proved in Item~\emph{(1)} in Lemma~\ref{Lemma:StrongConsequences}. 
\begin{figure}[h!]
\vspace{1cm}
\centering
\begin{tikzcd}
 {\text{UGAS w.r.t.~}  {\mathcal{S}}_{\text{dw}}(\tau)}&{}& {\text{UGAS w.r.t.~}  {\mathcal{S}}_{\text{adw}}(\tau,N_0),\;\;\forall\;N_0\in \N } & {} & {\GAS} \\
 {\exists \;\text{ LFs as in~Prop.~\ref{prop:ConverseDwellTime}}}&{}& {} && {\exists \;\text{ LFs as in~Prop.~\ref{prop:AneelConditions}}}
  \arrow["\text{ Subsec.~\ref{subsec:Counter}}"{pos=0.6},shift left=2, Leftarrow, from=1-3, to=1-1,"\;/" marking]
  \arrow["\text{Lemma~\ref{Lemma:StrongConsequences}}"{pos=0.4},shift left=1, shorten <=6pt, shorten >=18pt, Rightarrow, from=1-5, to=1-3]
  \arrow["{\text{(?) Conj.~\ref{conj:Conje}}}"{pos=0.55}, shift left=2, shorten <=18pt, shorten >=6pt,dashed, from=1-3, to=1-5, " \text{ }" marking]{rr}{}
  \arrow["{\text{Theorem~\ref{thm:Converse}}}"', shift right=3, Leftrightarrow, from=2-5, to=1-5]
  \arrow["{\text{ Prop.~\ref{prop:ConverseDwellTime}}}"{}, shift left=0.5, Leftrightarrow, from=1-1, to=2-1]
  \arrow["",shift right=-2, Leftarrow, from=1-1, to=1-3]
\end{tikzcd}
\caption{Schematic collection of the main results for the nonlinear case. The notation ``LFs'' stands for ``Lyapunov functions''. The linear-case scheme is completely equivalent, \emph{mutatis mutandis}, replacing  UGAS by $\text{UGES}_\rho$ and $\GAS$ by $\ES$.}  
\vspace{0.4cm}
    \label{fig:dfsfds}
\end{figure}
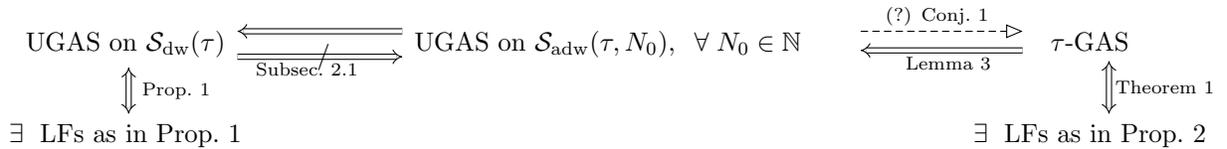

Summarizing our results, we highlight the relations between different properties in Figure~\ref{fig:dfsfds}. One can see that, as a by-product, we observe that the existence of functions as in Proposition~\ref{prop:AneelConditions} (resp. norms as in Corollary~\ref{cor:ConverseLinear} for the linear case) implies the existence of functions (resp. norms) as in~Proposition~\ref{prop:ConverseDwellTime}.

To conclude, this manuscript characterized stability for switched nonlinear systems under average dwell-time constraints, establishing necessary and sufficient conditions in terms of multiple Lyapunov functions. The performed analysis highlighted the presence of a strict gap between stability for dwell-time switching signals and stability for \emph{average} dwell-time constrained switching signals, as demonstrated through a counterexample. Building on these insights, we developed a converse result for average dwell-time constrained systems, presenting inequalities independent of the subsystem flow maps, thereby facilitating easier verification. Additionally, we examined the particular case of linear switched systems, deriving a corollary from our main result. This study enhances the theoretical understanding of stability in switched nonlinear systems, offering valuable insights with potential implications for practical applications.

\section*{Acknowledgements}
 The authors would like to thank anonymous reviewers for their helpful comments in improving the presentation of the manuscript. The authors are also grateful to  Sigurdur Hafstein for his help with the numerical simulations reported in Figure~\ref{fig:Parameters}.

\bibliography{biblio} 
 \bibliographystyle{plain}

\appendix

\section{Concatenation of Switching Signals}\label{app:Concat}
In this section, we collect basic definitions and results concerning concatenation of signals.
\begin{defn}[Concatenation]
Given any finite index set $\cI$, consider $\sigma, \gamma \in \cS$ defined in~\eqref{eq:arbitrarySwitching}. Given any $\delta> 0$ we define $\sigma \diamond_\delta\gamma\in \cS$ by
\[
\sigma \diamond_\delta \gamma(t):=\begin{cases}
\sigma(t)\;\;\;\;&\text{if }t<\delta,\\
\gamma(t-\delta)\;\;\;\;&\text{if }t\geq \delta.
\end{cases}
\]
\end{defn}
We now provide bounds on the number of discontinuity points of a concatenation of signals.
\begin{lemma}\label{lemma:Concatenation}
Let us consider $\sigma,\gamma\in \cS$. For any $\delta>0$ and any $0\leq s\leq t$, it holds that
\begin{equation}\label{eq:Concatenation1}
 N_{\sigma\diamond_\delta\gamma}(s,t) \leq N_\sigma(s,\delta)+N_\gamma(0, t-\delta).
\end{equation}
Moreover, considering any $0\leq \delta\leq t$, the following holds:
\begin{equation}\label{eq:Concatenation2}
 N_{\sigma\diamond_\delta\gamma}(0,t)=\begin{cases}
N_\sigma(0,\delta)+N_\gamma(0, t-\delta)-1\;\;\;&\text{if } \lim_{r\to \delta^-}\sigma(r)=\gamma(0),\\
N_\sigma(0,\delta)+N_\gamma(0, t-\delta)\;\;\;&\text{if } \lim_{r\to \delta^-}\sigma(r)\neq \gamma(0).
 \end{cases}
\end{equation}

\end{lemma}
\begin{proof}
Let us prove~\eqref{eq:Concatenation1}, we consider $s\leq t$. The cases $s\leq t\leq \delta$ and $\delta\leq s\leq t$  are straightforward. Indeed, suppose $s\leq t\leq \delta$, we have $N_{\sigma\diamond_\delta \gamma}(s,t)=N_{\sigma}(s,t)$. In the case $\delta\leq s\leq t$ we have $N_{\sigma\diamond_\delta \gamma}(s,t)=N_{\gamma}(s-\delta,t-\delta)\leq N_\gamma(0,t-\delta)$. Suppose then $s<\delta<t$, from the cases studied above, we have
\[
N_{\sigma\diamond_\delta \gamma}(s,t)=N_{\sigma\diamond_\delta \gamma}(s,\delta)+N_{\sigma\diamond_\delta \gamma}(\delta,t)\leq N_\delta(s,\delta)+N_\gamma(0,t-\delta)
\]
concluding the proof of~\eqref{eq:Concatenation1}.
Let us prove~\eqref{eq:Concatenation2} by considering $0\leq \delta\leq t$. Let us suppose first that $\lim_{r\to \delta^-}\sigma(r)=\gamma(0)$ i.e., in a open left neighborhood of $\delta$ the signal $\sigma$ is equal to $\gamma(0)$, then
\[
N_{\sigma\diamond_\delta \gamma}(0,t)=N_{\sigma\diamond_\delta \gamma}(0,\delta)+N_{\sigma\diamond_\delta \gamma}(\delta,t)=N_{\sigma}(0,\delta)+N_{\gamma}(0,t-\delta)-1
\]
where the term $-1$ has been added since $t_0=0$ is a discontinuity point of $\gamma$ (and thus taken into account in $N_{\gamma}(0,t-\delta)$), while $\delta$ is not a discontinuity point of $\sigma\diamond_\delta \gamma$. The case $\lim_{r\to \delta^-}\sigma(r)\neq\gamma(0)$ is similar and thus left to the reader.
\end{proof}

We note that~\eqref{eq:Concatenation1} in Lemma~\ref{lemma:Concatenation} in particular proves that $\cS^{\infty}_{\ad}(\tau)$ is closed under arbitrary concatenations: if $\sigma_1\in  \cS_{\ad}(\tau,N_1)$  and $\sigma_2\in  \cS_{\ad}(\tau,N_2)$ for some $N_1,N_2\in \N$, then $\sigma_1 \diamond_\delta \sigma_2\in  \cS_{\ad}(\tau,N_1+N_2)$, for all $\delta\geq0$.
 On the other hand, $\cS_{\ad}(\tau,N_0)$ is not closed under concatenation, for any $N_0\in \N$.

\section{Proof of local Lipschitz property}\label{app:LocLip}
In this Appendix we provide the proof of local Lipschitz continuity of the functions $W_1,\dots, W_{\rm m}:\R^n\to \R$ constructed in the proof of our main converse Lyapunov result, i.e., Theorem~\ref{thm:Converse}.
\begin{lemma}\label{lemma:LocLipschitz}
Under the hypotheses of Theorem~\ref{thm:Converse}, the functions $W_i:\R^n\to \R$ defined in~\eqref{eq:DefinitionV_i} are continuous and locally Lipschitz on $\R^n\setminus \{0\}$.
\end{lemma}
\begin{proof}[Sketch of the proof]
The proof follows by adapting the arguments in~\cite{TeelPraly2000} and~\cite{LinSontag96} to our case.
We first assume for simplicity that, for the dynamical systems defined by $f_1,\dots f_{\rm m}$, the $\{0\}$ equilibrium is not reached in finite time, i.e., for any $x\in \R^n\setminus \{0\}$, for any $t\in \R_+$, we have
$\Phi_i(t,x)\neq 0$. This assumption simplifies the proof but can be easily relaxed as we suggest at the end of the proof.\\
Let us fix $i\in \cI$, and let us recall that we defined $W_i:\R^n\to \R$ by
\[
W_i(x):=\sup_{\sigma\in \cS_i}\sup_{s\geq 0} \frac{e^{\gamma s}}{e^{\alpha \tau N_\sigma(0,s)}}\alpha_1(|\Phi_\sigma(s,x)|).
\]
We have already proved that $W_i$ satisfies~\eqref{eq:Sandwich-NonLinearBB} with $\wt \alpha_1=\frac{1}{e^{\alpha\tau}}\eta_1^{-1}$ and $\wt \alpha_2\equiv \eta_2$, for certain $\eta_2\in \cK_\infty$ and $\eta_1\in \cK_\infty\cap \cC^1(\R_+\setminus \{0\},\R)$ such that $\eta_1'(s)>0$ for all $s\in \R_+\setminus \{0\}$.
First of all, for any $x\in \R^n\setminus \{0\}$, let us introduce
\begin{equation}\label{eq:TUno}
T_1(x):=\varepsilon^{-1}\left(1-\ln\Big(\frac{W_i(x)}{\eta_2(|x|)}\Big)\right ),
\end{equation}
then, we have 
\begin{equation}\label{eq:FiniteIntervalReduction}
W_i(x)=\sup_{\sigma\in \cS_i}\max_{s\in [0,T_1(x)]} \frac{e^{\gamma s}}{e^{\alpha \tau N_\sigma(0,s)}}\alpha_1(|\Phi_\sigma(s,x)|).
\end{equation}
Indeed, recalling that by~\eqref{eq:BoundSolutionProof} we have $\frac{e^{\gamma t}}{e^{\alpha \tau N_\sigma(0,t)}}\alpha_1(|\Phi_\sigma(t,x)|)\leq \eta_2(|x|)e^{-\varepsilon t}$, $\forall \sigma \in \cS,\;\forall x\in \R^n$ and $\forall t\in \R_+$, we obtain
\[
\begin{aligned}
W_i(x)&=\max\left\{\sup_{\sigma\in \cS_i}\max_{s\in[0,T_1(x)]} \frac{e^{\gamma s}}{e^{\alpha \tau N_\sigma(0,s)}}\alpha_1(|\Phi_\sigma(s,x)|),\;\;\sup_{\sigma\in \cS_i}\sup_{s\geq T_1(x)} \frac{e^{\gamma s}}{e^{\alpha \tau N_\sigma(0,s)}}\alpha_1(|\Phi_\sigma(s,x)|) \right \}
\\& \leq \max\left\{\sup_{\sigma\in \cS_i}\max_{s\in[0,T_1(x)]} \frac{e^{\gamma s}}{e^{\alpha \tau N_\sigma(0,s)}}\alpha_1(|\Phi_\sigma(s,x)|),\;\;\eta_2(|x|)e^{-\varepsilon T_1(x)} \right \}
\\& \leq \max\left\{\sup_{\sigma\in \cS_i}\max_{s\in[0,T_1(x)]} \frac{e^{\gamma s}}{e^{\alpha \tau N_\sigma(0,s)}}\alpha_1(|\Phi_\sigma(s,x)|),\;\;\frac{1}{e}W_i(x) \right \},
\end{aligned}
\]
from which we obtain~\eqref{eq:FiniteIntervalReduction}. Given an arbitrary $x\in \R^n\setminus \{0\}$, we consider a compact neighborhood of $x$, denoted by $\cU$ and we suppose $\cU\subset \R^n\setminus \{0\}$.
Let us denote by 
\[
T_\cU=\max_{y\in \cU}\varepsilon^{-1}\left(1-\ln\Big(\frac{\wt \alpha_1(|y|)}{\eta_2(|y|)}\Big)\right )
\] 
which is well-defined by continuity of $\wt \alpha_1$ and $\eta_2$.
Since,  by~\eqref{eq:Sandwich-NonLinearBB} $W_i(z)\geq \wt \alpha_1(|z|)$ for every $z\in \R^n$ and recalling~\eqref{eq:TUno}, we have that $T_1(y)\leq T_\cU$ for every $y\in \cU$. Let us then define
\[
K=\overline {\cR^{\leq T_\cU}(\cU)}:=\overline{\{z=\Phi_\sigma(t,y)\;\vert\;\sigma\in \cS_i,\;y\in \cU,\;\;t\in [0,T_\cU]\}}.
\]
By forward completeness and by the hypothesis that $0$ is never reached in finite time, it can be proved that $K\subset \R^n\setminus \{0\}$ is compact, see also~\cite[Proposition 5.1]{LinSontag96}. For any $j\in \cI$, let us define $L_j\geq 0$ as the Lipschitz constant of $f_j$ in $K$ and consider $L:=\max_{j\in \cI}L_j$. By continuous dependence of initial conditions (see \cite[Theorem 3.4]{khalil2002nonlinear}), we have that, for every $y\in \cU$, any $\sigma\in \cS_i$ and any $t\leq T_\cU$, we have
\[
|\Phi_\sigma(t,y)-\Phi_\sigma(t,x)|\leq e^{Lt}|y-x|\leq e^{LT_\cU}|y-x|.
\]
From now on, we define $R=e^{LT_\cU}$; moreover, let us call $Q:=\max_{y\in K}\alpha_1'(|y|)$; we have $Q\in (0,+\infty)$ since $\alpha_1=\eta_1^{-1}\in \cK_\infty\cap \cC^1(\R_+\setminus \{0\},\R)$. Let us consider any $y\in \cU$, we have
\[
\begin{aligned}
W_i(y)&=\sup_{\sigma\in \cS_i}\max_{s\in [0,T_\cU]} \frac{e^{\gamma s}}{e^{\alpha \tau N_\sigma(0,s)}}\alpha_1(|\Phi_\sigma(s,y)|)\leq\sup_{\sigma\in \cS_i}\max_{s\in [0,T_\cU]} \frac{e^{\gamma s}}{e^{\alpha \tau N_\sigma(0,s)}}\left (\alpha_1(|\Phi_\sigma(s,x)|) +\alpha_1'(|\Phi_\sigma(s,x)|)R|x-y|\right) \\&\leq \sup_{\sigma\in \cS_i}\max_{s\in [0,T_\cU]} \frac{e^{\gamma s}}{e^{\alpha \tau N_\sigma(0,s)}}\left (\alpha_1(|\Phi_\sigma(s,x)|)+QR|x-y|\right)\\
&\leq  \sup_{\sigma\in \cS_i}\max_{s\in [0,T_\cU]} \frac{e^{\gamma s}}{e^{\alpha \tau N_\sigma(0,s)}}\alpha_1(|\Phi_\sigma(s,x)|)\,+\,e^{\gamma T_\cU}QR|x-y|=W_i(x)+A|x-y|
\end{aligned}
\]
with $A=e^{\gamma T_\cU}QR$. With a similar reasoning, one can conclude that $W_i(y)\geq W_i(x)-A|x-y|$, proving that $W_i$ is Lipschitz continuous in $\cU$.
By arbitrariness of $x\in \R^n\setminus \{0\}$ we conclude that $W_i$ is locally Lipschitz on $\R^n\setminus \{0\}$.
Continuity then trivially follows by~\eqref{eq:Sandwich-NonLinearBB}.
\\
In the case where the equilibrium $\{0\}$ is possibly reached in finite time, one has to consider, given any $\sigma\in \cS_i$, the time
\[
T_2(x,\sigma)=\sup\left \{t\geq 0\;:\;\frac{e^{\gamma t}}{e^{\alpha \tau N_\sigma(0,s)}}\alpha_1(|\Phi_\sigma(s,y)|)>\alpha_1(|x|)\right \},
\] 
with the convention that $\sup\emptyset= 0$. If $\left \{t\geq 0\;:\;\frac{e^{\gamma t}}{e^{\alpha \tau N_\sigma(0,s)}}\alpha_1(|\Phi_\sigma(s,y)|)>\alpha_1(|x|)\right \}\neq \emptyset$, then $T_2(x,\sigma)$ is bounded (uniformly with respect to $\sigma$), since $\frac{e^{\gamma t}}{e^{\alpha \tau N_\sigma(0,s)}}\alpha_1(|\Phi_\sigma(s,y)|)\leq e^{-\varepsilon t}\eta_2(|x|)$, for any $\sigma\in \cS_i$, any $x\in \R^n\setminus \{0\}$, any $t\geq 0$. Then, we have
\[
W_i(x)=\sup_{\sigma\in \cS_i}\max_{s\in [0,T_2(x,\sigma)]} \frac{e^{\gamma s}}{e^{\alpha \tau N_\sigma(0,s)}}\alpha_1(|\Phi_\sigma(s,x)|),
\]
and adapting the reasoning from the previous case, the claim can be proved, \emph{mutatis mutandis}.
\end{proof}}
\end{document}